\documentclass[11pt]{article}
\usepackage{latexsym,amsfonts,amssymb,amsmath,amsthm}
\usepackage{graphicx}

\usepackage[usenames,dvipsnames]{color}
\usepackage{ulem,comment}

\usepackage{color}

\begin{document}

\newtheorem{tm}{Theorem}[section]
\newtheorem{prop}[tm]{Proposition}
\newtheorem{defin}[tm]{Definition} 
\newtheorem{coro}[tm]{Corollary}
\newtheorem{lem}[tm]{Lemma}
\newtheorem{assumption}[tm]{Assumption}
\newtheorem{rk}[tm]{Remark}
\newtheorem{nota}[tm]{Notation}
\numberwithin{equation}{section}

\newcommand{\stk}[2]{\stackrel{#1}{#2}}
\newcommand{\dwn}[1]{{\scriptstyle #1}\downarrow}
\newcommand{\upa}[1]{{\scriptstyle #1}\uparrow}
\newcommand{\nea}[1]{{\scriptstyle #1}\nearrow}
\newcommand{\sea}[1]{\searrow {\scriptstyle #1}}
\newcommand{\csti}[3]{(#1+1) (#2)^{1/ (#1+1)} (#1)^{- #1
 / (#1+1)} (#3)^{ #1 / (#1 +1)}}
\newcommand{\RR}[1]{\mathbb{#1}}

\newcommand{\rd}{{\mathbb R^d}}
\newcommand{\ep}{\varepsilon}
\newcommand{\rr}{{\mathbb R}}
\newcommand{\alert}[1]{\fbox{#1}}
\newcommand{\eqd}{\sim}
\def\p{\partial}
\def\R{{\mathbb R}}
\def\N{{\mathbb N}}
\def\Q{{\mathbb Q}}
\def\C{{\mathbb C}}
\def\l{{\langle}}
\def\r{\rangle}
\def\t{\tau}
\def\k{\kappa}
\def\a{\alpha}
\def\la{\lambda}
\def\De{\Delta}
\def\de{\delta}
\def\ga{\gamma}
\def\Ga{\Gamma}
\def\ep{\varepsilon}
\def\eps{\varepsilon}
\def\si{\sigma}
\def\Re {{\rm Re}\,}
\def\Im {{\rm Im}\,}
\def\E{{\mathbb E}}
\def\P{{\mathbb P}}
\def\Z{{\mathbb Z}}
\def\D{{\mathbb D}}
\newcommand{\ceil}[1]{\lceil{#1}\rceil}

\title{Parabolic-elliptic chemotaxis model with space-time dependent logistic sources on $\mathbb{R}^N$. III. Transition fronts}

\author{
Rachidi B. Salako\\
Department of Mathematics\\
The Ohio State University\\
Columbus OH, 43210-1174\\
 and \\
 Wenxian Shen\thanks{Partially supported by the NSF grant DMS--1645673} \\
Department of Mathematics and Statistics\\
Auburn University\\
Auburn University, AL 36849 }

\date{}
\maketitle

\begin{abstract}
The current work is the third of a series of three papers  devoted to the study of asymptotic dynamics in the following parabolic-elliptic chemotaxis system with space and time dependent logistic source,
\begin{equation}\label{main-eq-abstract}
\begin{cases}
\partial_tu=\Delta u -\chi\nabla\cdot(u\nabla v)+u(a(x,t)-b(x,t)u),\quad x\in\R^N,\cr
0=\Delta v-\lambda v+\mu u ,\quad x\in\R^N,
\end{cases}
\end{equation}
where $N\ge 1$ is a positive integer,  $\chi, \lambda$ and $\mu$ are positive constants,
 and the functions $a(x,t)$ and $b(x,t)$ are positive and bounded.  In the first of the series \cite{SaSh_6_I}, we studied the phenomena of pointwise and uniform persistence for solutions with strictly positive initials, and the asymptotic spreading for solutions with compactly supported or front like initials. In the second of the series \cite{SaSh_6_II}, we investigate the existence, uniqueness and stability of strictly positive entire solutions of \eqref{main-eq-abstract}. In particular, in the case of space homogeneous logistic source (i.e.
 $a(x,t)\equiv a(t)$ and $b(x,t)\equiv b(t)$), we proved in \cite{SaSh_6_II} that  the unique spatially homogeneous strictly positive entire solution $(u^*(t),v^*(t))$ of \eqref{main-eq-abstract}  is uniformly and exponentially stable with respect to strictly positive perturbations when $0<2\chi\mu<\inf_{t\in\R}b(t)$.

 In the current part of the series, we discuss the existence of transition front solutions of \eqref{main-eq-abstract} connecting $(0,0)$ and $(u^*(t),v^*(t))$ in the case of space homogeneous logistic source. We show that for every $\chi>0$ with $\chi\mu\big(1+\frac{\sup_{t\in\R}a(t)}{\inf_{t\in\R}a(t)}\big)<\inf_{t\in\R}b(t)$, there is a positive constant ${c}^{*}_\chi$ such that for every $\underbar{c}> {c}^*_{\chi}$ and every unit vector $\xi$, \eqref{main-eq-abstract} has a  transition front solution  of the form $(u(x,t),v(x,t))=(U(x\cdot\xi-C(t),t),V(x\cdot\xi-C(t),t))$ satisfying that $C'(t)=\frac{a(t)+\kappa^2}{\kappa}$ for some positive number
 $\kappa$,  $\liminf_{t-s\to\infty}\frac{C(t)-C(s)}{t-s}=\underline{c}$, and
 $$
 \lim_{x\to-\infty}\sup_{t\in\R}|U(x,t)-u^*(t)|=0 \quad \text{and}\quad \lim_{x\to\infty}\sup_{t\in\R}|\frac{U(x,t)}{e^{-\kappa x}}-1|=0.
 $$
 Furthermore, we prove that there is  no transition front solution $(u(x,t),v(x,t))=(U(x\cdot\xi-C(t),t),V(x\cdot\xi-C(t),t))$ of \eqref{main-eq-abstract}  connecting $(0,0)$ and $(u^*(t),v^*(t))$ with least mean speed less than $2\sqrt{\underbar{a}}$, where
 $\underline{a}=\liminf_{t-s\to\infty}\frac{1}{t-s}\int_{s}^{t}a(\tau)d\tau$.
\end{abstract}



\section{Introduction and the Statements of the Main Results}

The current work is the third part of a series of three papers on the asymptotic dynamics in the following parabolic-elliptic chemotaxis system with space-time dependent logistic source,
\begin{equation}\label{P}
\begin{cases}
u_t=\Delta u -\chi\nabla\cdot(u\nabla v)+u(a(x,t)-b(x,t)u),\quad x\in\R^N,\cr
0=\Delta v-\lambda v+\mu u ,\quad x\in\R^N,
\end{cases}
\end{equation}
where $u(x,t)$ and $v(x,t)$ denote mobile species density and chemical density functions, respectively,  $\chi$ is a positive constant which measures the sensitivity with respect to chemical signals, $a(x,t)$ and $b(x,t)$ are positive functions  and measure the  growth and self limitation of the mobile species, respectively. The constant $\mu$ is positive and the term $+\mu u$ in the second equation of \eqref{P} indicates that the mobile species produces the chemical substance over time. The positive constant $\lambda$ measures the degradation rate of the chemical substance.  System \eqref{P}  is a type of the celebrated parabolic-elliptic Keller-Segel   chemotaxis systems (see \cite{KeSe1, KeSe2})  with space-time dependent logistic source.

The objective of the series of the three papers  is to study the asymptotic dynamics in the chemotaxis system \eqref{P} on the whole space with space and/or time dependent logistic source.
In the first of the series, \cite{SaSh_6_I}, we studied the phenomena of pointwise and uniform persistence for solutions with strictly positive initials, and the asymptotic spreading in  \eqref{P} for solutions with compactly supported or front like initials. In the second part of the series, \cite{SaSh_6_II}, we investigated the existence, uniqueness and stability of strictly positive entire solutions of \eqref{P}. In particular it was shown in \cite{SaSh_6_II} that, if the logistic source is space homogeneous, in which case \eqref{P} becomes,
\begin{equation}\label{P00}
\begin{cases}
u_t=\Delta u -\chi\nabla\cdot(u\nabla v)+u(a(t)-b(t)u),\quad x\in\R^N,\cr
0=\Delta v-\lambda v+\mu u ,\quad x\in\R^N,
\end{cases}
\end{equation}
then for every $\chi>0$ with  $0<2\chi\mu<\inf_{t\in\R}b(t)$ there is positive number $\alpha_{\chi}>0$, such that for every positive initial function $u_0\in C^b_{\rm unif}(\R^N)=\{u\in C(\R^N)\,|\, u(x)$ is bounded and uniformly continuous on $\R^M\}$ with $\inf_{x\in\R^N}u_0(x)>0$,  there is  $M>0$ such that
$$
\|u(\cdot,t+t_0,t_0,u_0)-u^*(t)\|_\infty\leq Me^{-\alpha_\chi t}\,\, {\rm and}\,\, \|v(\cdot,t+t_0;t_0,u_0)-v^*(t)\|_{\infty}\leq \frac{\mu}{\lambda}Me^{-\alpha_\chi t}
$$
for all $t\ge 0$ and $t_0\in\R$,
where $(u(x,t;t_0,u_0,v_0),v(x,t;t_0,u_0))$ denote the unique classical solution of \eqref{P00} with
$$
\lim_{t\to 0^+}\|u(\cdot,t+t_0;t_0,u_0)-u_0(\cdot)\|_{\infty}=0,
$$
and $u^*(t)$ is the unique strictly positive entire solution of the Fisher-KKP equation
\begin{equation}\label{kpp-eq}
u_t=\Delta u +u(a(t)-b(t)u),
\end{equation}
and $v^*(t)=\frac{\mu}{\lambda}u^*(t)$. Hence, when $0<2\chi\mu<\inf_{t\in\R}b(t)$,   the unique spatially homogeneous strictly positive entire solution  $(u^*(t),v^*(t))$ of \eqref{P0} is uniformly and exponentially stable with respect to strictly positive perturbations.

Observe that with $N=1$, $a(t)\equiv 1$ and $b(t)\equiv 1$, \eqref{kpp-eq} becomes
\begin{equation}
\label{fisher-kpp}
u_t=u_{xx}+u(1-u),\quad x\in\R.
\end{equation}
Equation \eqref{fisher-kpp}  is called in literature  Fisher-KPP   equation
 due to the pioneering works of Fisher
\cite{Fisher} and Kolmogorov, Petrowsky, Piskunov \cite{KPP} on traveling wave solutions and take-over properties of \eqref{fisher-kpp}.
Fisher in
\cite{Fisher} found traveling wave solutions $u(t,x)=\phi(x-ct)$ of \eqref{fisher-kpp}
$(\phi(-\infty)=1,\phi(\infty)=0)$ of all speeds $c\geq 2$ and
showed that there are no such traveling wave solutions of slower
speed. He conjectured that the take-over occurs at the asymptotic
speed $2$. This conjecture was proved in \cite{KPP}  {for some special initial distribution and was proved in \cite{ArWe2} for the general case.
 More precisely, it is proved
in \cite{KPP} that for the  nonnegative solution $u(t,x)$ of \eqref{fisher-kpp} with
$u(0,x)=1$ for $x<0$ and $u(0,x)=0$ for $x>0$, $\lim_{t\to \infty}u(t,ct)$ is $0$ if $c>2$ and $1$ if $c<2$. It is proved
in \cite{ArWe2} that for any
nonnegative solution $u(t,x)$ of (\ref{fisher-kpp}), if at
time $t=0$, $u$ is $1$ near $-\infty$ and $0$ near $\infty$, then
$\lim_{t\to \infty}u(t,ct)$ is $0$ if $c>2$ and $1$ if $c<2$.}
In
literature, $c^*=2$ is   called the {\it
spreading speed} for \eqref{fisher-kpp}.

A huge amount of research has been carried out toward various extensions of
 traveling wave solutions and take-over properties  of \eqref{fisher-kpp} to general time and space independent
as well as time and/or space dependent Fisher-KPP type  equations.
See, for example,  \cite{ArWe1}, \cite{ArWe2}, \cite{Bra}, \cite{Ham},  \cite{Kam}, \cite{Sat}, \cite{Uch}, etc., for
the extension to general time and space independent Fisher-KPP type equations;  see
 \cite{BeHaNa1,  BeHaRo,  FrGa,   HuZi1,
LiYiZh, LiZh, LiZh1,  Nad,  NoRuXi, NoXi, SaSh_7_II, Wei1, Wei2},  and references therein for
the extension to time and/or space periodic Fisher-KPP
type equations; and see
\cite{BeHa07, BeHa12,  HePaSt, HuSh,  Mat, Nad1,  NaRo1, NaRo2, NaRo3, NoRoRyZl,  She6, She7, She8, TaZhZl, Xin1, Zla}, and references therein for
the extension to quite general time and/or space dependent Fisher-KPP
type equations. It should be pointed out that the so called periodic traveling wave solutions or pulsating traveling fronts to time and/or space periodic
reaction diffusion equations are
 natural extension of the notion of traveling wave solutions in the classical sense,  and that  the so called transition fronts or generalized traveling waves
 to general time and/or space dependent reaction equations are the natural extension of the notion of traveling wave solutions in the classical sense
 (see  \cite{BeHa07, BeHa12} for the introduction of the notion of transition fronts or generalized traveling waves  in the general case, and
  \cite{Mat, She4, She7, She8} for the time almost periodic or space almost periodic cases).

   Considering a chemotaxis model  on the whole space, it is important to study the spatial spreading and propagating properties of the mobile species in the model. Transition front solutions or generalized traveling wave solutions  and spatial spread speeds are among those used to characterize such properties. There are many studies on traveling wave solutions of various types of chemotaxis models, see, for
example, \cite{AiHuWa, AiWa, FuMiTs, HoSt, LiLiWa, MaNoSh, NaPeRy, SaSh2, SaSh3, Wan}, etc.. It should be mentioned that
 spreading speeds and traveling wave solutions of \eqref{P} with $a(x,t)$ and $b(x,t)$ being constant functions are studied in \cite{HaHe, NaPeRy, SaSh2, SaSh3}. When $a(x,t)$ and $b(x, t)$ depend on $x$ and $t$,  as it is mentioned in the above there are many studies on spreading speeds and transition front solutions of \eqref{P} with $\chi=0$, but there is little study on transition front solutions  of \eqref{P} with $\chi\not =0$.

 The objective of this third part of the  series is to study the existence of transition front solutions of \eqref{P00}, i.e., \eqref{P} in the case that $a(x,t)\equiv a(t)$ and $b(x,t)\equiv b(t)$, connecting $(0,0)$ and $(u^*(t),v^*(t))$.
 To be more precise,  we study the existence of positive entire solutions of \eqref{P00} with  the form
 $(u(x,t),v(x,t))=(U(x\cdot\xi-C(t),t),V(x\cdot\xi-C(t),t))$ for some $\xi\in S^{N-1}=\{\xi\in\R^N\,|\, \|\xi\|=1\}$ and some $C(t)$, where
 $(U(-\infty,t),V(-\infty,t))=(u^*(t),v^*(t))$ and $(U(\infty,t),V(\infty,t))=(0,0)$. It is not difficult to see that, if $(u(x,t),v(x,t))=(U(x\cdot\xi-C(t),t),V(x\cdot\xi-C(t),t))$ $(x\in\R^N$) is an entire solution of \eqref{P00}, then
 $(u(x,t),v(x,t))=(U(x-C(t),t),V(x-C(t),t))$ $(x\in\R)$ is an entire solution of
 \begin{equation}\label{P0}
\begin{cases}
u_t=u_{xx} -\chi(u v_x)_x+u(a(t)-b(t)u),\quad x\in\R,\cr
0=v_{xx}-\lambda v+\mu u ,\quad x\in\R.
\end{cases}
\end{equation}
 We will then study the existence of transition front solutions of \eqref{P0} connecting $(0,0)$ and $(u^*(t),v^*(t))$.

In the rest of the introduction, we introduce notations and standing assumptions,  and state the main results of the current paper.

\subsection{Notations and standing assumptions}

For every function $w : \R\times I\to \R$, where $I\subset \R$, we set $w_{\inf}(t)=\inf_{x\in\R}w(x,t)$, $w_{\sup}(t)=\sup_{x\in\R}w(x,t)$, $w_{\inf}=\inf_{(x,t)\in \R\times I}w(x,t)$, and $w_{\sup}=\sup_{(x,t)\in \R\times I}w(x,t)$.
Let
$$
C_{\rm unif}^b(\R)=\{u\in C(\R)\,|\, u(x)\,\,\, \text{is uniformly continuous in}\,\,\, x\in\R\,\,\, \text{and}\,\, \sup_{x\in\R}|u(x)|<\infty\}
$$
equipped with the norm $\|u\|_\infty=\sup_{x\in\R}|u(x)|$. For any $0\le \nu<1$, let
$$
C_{\rm unif}^{b,\nu}(\R)=\{u\in C_{\rm unif}^b(\R)\,|\, \sup_{x,y\in\R,x\not = y}\frac{|u(x)-u(y)|}{|x-y|^\nu}<\infty\}
$$
with norm $\|u\|_{C_{\rm unif}^{b,\nu}}=\sup_{x\in\R}|u(x)|+\sup_{x,y\in\R,x\not =y}\frac{|u(x)-u(y)|}{|x-y|^\nu}$. Hence $C_{\rm unif}^{b,0}(\R)=C_{\rm unif}^{b}(\R)$.

For given $f\in L_{\rm loc}^1(\R)$, let
$$
\underline{f}=\liminf_{t-s\to\infty}\frac{1}{t-s}\int_{s}^{t}f(\tau)d\tau
\quad {\rm and}\quad
\overline{f}=\limsup_{t-s\to\infty}\frac{1}{t-s}\int_{s}^{t}f(\tau)d\tau.
$$
$\underline{f}$ and $\overline{f}$  are called the
 {\it least mean} and {\it greatest mean} of $f$,
 respectively.

Throughout the remaining of this paper, we shall always suppose that the following standing assumption holds.

\medskip

\noindent {\bf (H)} {\it $a(x,t)\equiv a(t)$ and  $b(x,t)\equiv b(t)$ are uniformly H\"older continuous in $t\in\R$ with exponent $0<\nu_0<1$
and
$$
 0<\inf_{t\in\R}\min\{a(t), b(t)\} \leq \sup_{t\in\R}\max\{a(t),b(t)\}<\infty.
$$
}

Observe from ${\bf (H)} $ that
\begin{equation}
0<a_{\inf}\leq \underline{a}\le \overline{a}\leq a_{\sup}<\infty.
\end{equation}

\subsection{Main results}

 For given $u_0\in C_{\rm unif}^b(\R)$ and $t_0\in\R$, let $(u(x,t;t_0,u_0),v(x,t;t_0,u_0))$ be the classical solution of \eqref{P0} with $u(x,t_0;t_0,u_0)=u_0(x)$ for every $x\in\R$ (see \cite[Theorem 1.1]{SaSh1} for the existence of $(u(x,t;t_0,u_0),v(x,t;t_0,u_0))$). Note that if $u_0(x)\geq 0$ then $u(x,t;t_0,u_0)\geq 0$ and $v(x,t;t_0,u_0)\ge 0$ for every $x\in\R$ and $t\in[t_0, t_0+T_{\max})$, where $[t_0, t_0+T_{\max})$ denotes the maximal interval of existence of $(u(x,t;t_0,u_0),v(x,t;t_0,u_0))$. A classical solution $(u(x,t),v(x,t))$ of \eqref{P0} is said to be an {\it entire solution} of \eqref{P0} if it is defined for every $t\in\R$. Note that $(u(x,t),v(x,t))=(0,0)$ is an equilibrium solution of \eqref{P0}. Throughout this work we shall denote by $u^*(t)$ the unique strictly positive entire solution  of \eqref{kpp-eq} and $v^*(t)=\frac{\mu}{\lambda}u^*(t)$. Then $(u^*(t),v^*(t))$ is  a positive entire solution of \eqref{P0}.

 An entire solution $(u(x,t),v(x,t))$ of \eqref{P0} is called a {\it transition front solution} connecting $(0,0)$ and $(u^*(t),v^*(t))$ if
 \begin{equation}
 (u(x,t),v(x,t))=(U(x-C(t),t),V(x-C(t),t))
 \end{equation}
 for some $U(\cdot,\cdot)$ and $C(t)$ satisfying
 \begin{equation}
\lim_{z\to-\infty}|U(z,t)-u^*(t)|=\lim_{z\to-\infty}|V(z,t)-v^*(t)|=0, \quad \text{uniform in } t\in\R,
\end{equation}
and
\begin{equation}
\lim_{z\to\infty}U(z,t)=\lim_{z\to\infty}V(z,t)=0, \quad \text{uniform in } t\in\R.
\end{equation}
Let
$$
\underbar{c}=\liminf_{t-s\to\infty}\frac{C(t)-C(s)}{t-s}.
$$
 The function $(U(\cdot,\cdot),V(\cdot,\cdot))$ and $\underline{c}$  are called the {\it profile} and {\it least mean speed}, respectively, of the {transition front solution $(u(x,t),v(x,t))=(U(x-C(t),t),V(x-C(t),t))$}. If we suppose that $C(t)$ is of class $C^1$ and set $c(t)=C'(t)$, then
 $(U(\cdot,\cdot),V(\cdot,\cdot))$ and $c(\cdot)$ satisfy
\begin{equation}
\label{P0-1}
\begin{cases}
U_t=U_{xx} +c(t)U_x-\chi(U V_x)_{x}+U(a(t)-b(t)U),\quad x\in\R, \ t\in\R, \\
0=V_{xx} -\lambda V +\mu U,\quad x\in\R, \ t\in\R.
\end{cases}
\end{equation}

 Note that the function $\eta(\kappa):=\frac{\kappa(\sqrt{\lambda-\kappa^2}+\kappa)}{\lambda-\kappa^2}$
is strictly increasing on $(0, \sqrt{\lambda})$. For given $\chi> 0$ with $b_{\inf}>\chi\mu$, let
$\kappa_\chi\in (0, \sqrt{\lambda})$ be such that
\begin{equation}
\label{min-kappa-eq}
\frac{b_{\inf}-\chi\mu}{\chi\mu}= \frac{\kappa_\chi(\sqrt{\lambda-\kappa_\chi^2}+\kappa_\chi)}{\lambda-\kappa_\chi^2}.
\end{equation}
Then for any $0<\kappa\le \kappa_{\chi}$,
\begin{equation}
\label{min-kappa-eq1}
\frac{b_{\inf}-\chi\mu}{\chi\mu}\ge  \frac{\kappa(\sqrt{\lambda-\kappa^2}+\kappa)}{\lambda-\kappa^2}.
\end{equation}
Define
\begin{equation}
\label{minimal-speed-eq1}
{c}_\chi^*=\frac{\underline{a}+{\kappa}^2_{\chi^*}}{{\kappa}_{\chi^*}},
\end{equation}
 where ${\kappa}_\chi^{*}=\min\{\kappa_\chi, \sqrt{\underline{a}}\}$.
Let {\bf (H1)} be the following standing assumption.

\medskip

\noindent {\bf (H1)}  {\it $b_{\inf}>\chi\mu(1+\frac{a_{\sup}}{a_{\inf}})$.}

\medskip

The main results on the existence of transition front solutions of \eqref{P0} are stated in the following theorem.

\begin{tm}\label{Main-tm1}
Suppose that {\bf (H1)} holds.
\begin{itemize}
\item[(1)] For every $\underline{c}> {c}^*_\chi$, \eqref{P0} has a transition front solution $(u(x,t),v(x,t))=(U(x-C(t),t),V(x-C(t),t))$  connecting $(0,0)$ and $(u^*(t),v^*(t))$ with least mean speed $\underline{c}$. Furthermore, it holds that $C(t)=\int_0^{t}\frac{a(s)+\kappa^2}{\kappa}ds$  and
\begin{equation}\label{TW-eq}
\lim_{x\to-\infty}\sup_{t\in\R}|U(x,t)-u^*(t)|=0, \quad \text{and}\quad \lim_{x\to\infty}\sup_{t\in\R}|\frac{U(x,t)}{e^{-\kappa x}}-1|=0,
\end{equation}
where  $\kappa\in(0, \min\{\kappa_\chi,\sqrt{\underbar{a}}\})$ is such that $\underline{c}=\frac{\underline{a}+\kappa^2}{\kappa}$.

\item[(2)] If $a(t)$ and $b(t)$ are periodic in $t$ with period $T$, then for every $c> {c}^*_\chi$, \eqref{P0} has a
periodic transition front solution $(u(x,t),v(x,t))=(U(x-ct,t),V(x-ct,t))$  connecting $(0,0)$ and $(u^*(t),v^*(t))$, and satisfying
\begin{equation}
\label{TW-eq1}
U(x,t+T)=U(x,t),\quad V(x,t+T)=V(x,t),
\end{equation}
and
\begin{equation}\label{TW-eq2}
\lim_{x\to-\infty}\sup_{t\in\R}|U(x,t)-u^*(t)|=0, \quad \text{and}\quad \lim_{x\to\infty}\sup_{t\in\R}|\frac{U(x,t)}{e^{-\kappa (x+ct-\int_0^t c_\kappa(s)ds)}}-1|=0,
\end{equation}
where  $\kappa\in(0, \min\{\kappa_\chi,\sqrt{\hat{a}}\})$ $(\hat a=\frac{1}{T}\int_0^T a(t)dt$)  satisfies $c=\frac{\hat{a}+\kappa^2}{\kappa}$, and $c_\kappa(s)=\frac{a(s)+\kappa^2}{\kappa}$.

\item[(3)]  If $a(t)\equiv a$ and $b(t)\equiv b$ are independent of $t$, then for every $c> {c}^*_\chi$, \eqref{P0} has a
traveling wave solution $(u(x,t),v(x,t))=(U(x-ct),V(x-ct))$  connecting $(0,0)$ and $(u^*,v^*)=(\frac{a}{b}, \frac{\mu}{\lambda}\frac{a}{b})$, and satisfying
\begin{equation}\label{TW-eq3}
\lim_{x\to-\infty}|U(x)-u^*|=0, \quad \text{and}\quad \lim_{x\to\infty}|\frac{U(x)}{e^{-\kappa x}}-1|=0,
\end{equation}
where  $\kappa\in(0, \min\{\kappa_\chi,\sqrt{{a}}\})$ satisfies $c=\frac{{a}+\kappa^2}{\kappa}$.
\end{itemize}
\end{tm}

We have the following theorem on the nonexistence of transition front solutions of \eqref{P0}.

\begin{tm}\label{Main-tm2}
Suppose that {\bf (H1)} holds.
For every $\underline{c}< 2\sqrt{\underbar{a}}$, \eqref{P0} has no transition front solution $(u(x,t),v(x,t))=(U(x-C(t),t),V(x-C(t),t))$  connecting $(0,0)$ and $(u^*(t),v^*(t))$ with least mean speed $\underline{c}$.
\end{tm}

\begin{rk}
\label{TW-rk}
\begin{itemize}
\item[(1)] The results in Theorem \ref{Main-tm1}(1), (2), and Theorem \ref{Main-tm2}   are new. The result in  Theorem \ref{Main-tm1}(3) extends the results in
\cite{SaSh2} for the case $\lambda=\mu=1$ (see \cite[Theorem A and Remark 1.1]{SaSh2}).

 \item[(2)]  By Theorem \ref{Main-tm1}(1), for every $\xi\in S^{N-1}=\{\xi\in\R^N\,|\, \|\xi\|=1\}$ and every $\underline{c}> {c}^*_\chi$, \eqref{P} has a transition front solution $(u(x,t),v(x,t))=(U(x\cdot\xi-C(t),t),V(x\cdot\xi-C(t),t))$  connecting $(0,0)$ and $(u^*(t),v^*(t))$ with least mean speed $\underline{c}$, where $C(t)=\int_0^{t}\frac{a(s)+\kappa^2}{\kappa}ds$, $\kappa\in(0, \min\{\kappa_\chi,\sqrt{\underbar{a}}\})$ is such that $\underline{c}=\frac{\underline{a}+\kappa^2}{\kappa}$, and $(U(x,t),V(x,t))$ satisfies \eqref{TW-eq}.

\item[(3)] Let $c_0^*=2\sqrt {\underbar{a}}$. It is proved in  \cite[Theorem 2.3]{NaRo1} that $c_0^*$ is the minimal least mean speed of transition front solutions of \eqref{kpp-eq}, i.e.,  \eqref{P0} in the absence of the chemotaxis,  in the sense that for any $\underbar{c}>c_0^*$, \eqref{kpp-eq} has a transition front solution connecting $0$ and $u^*(t)$ with least mean speed $\underbar{c}$, and \eqref{kpp-eq} has no
    transition front solutions connecting $0$ and $u^*(t)$  with least mean speed smaller than $c_0^*$.

 \item[(4)] For fixed $\chi>0$ with $b_{\inf}>\chi\mu$,
  when the degradation rate $\lambda$ of the chemical substance is sufficiently large, we have $\kappa_\chi\ge \sqrt{\underbar{a}}$ and $\kappa_\chi^*=\sqrt {\underbar{a}}$, hence $c_\chi^*=2\sqrt {\underbar{a}}$, which is  the minimal least mean speed of transition front solutions of \eqref{kpp-eq}.
 Indeed, since the function $\lambda\mapsto \frac{\sqrt{\underline{a}}\left(\sqrt{\lambda-\underline{a}} +\sqrt{\underline{a}}\right)}{\lambda-\underbar{a}}$ is strictly decreasing and satisfies
 $$\lim_{\lambda\to \underline{a}+}\frac{\sqrt{\underline{a}}\left(\sqrt{\lambda-\underline{a}} +\sqrt{\underline{a}}\right)}{\lambda-\underline{a}}=\infty\,\, {\rm and}\,\, \lim_{\lambda\to\infty}\frac{\sqrt{\underline{a}}\left(\sqrt{\lambda-\underline{a}} +\sqrt{\underline{a}}\right)}{\lambda-\underbar{a}}=0,
  $$
  there is a unique $\lambda_\chi>\underline{a}$ such that
  $$
  \frac{\sqrt{\underline{a}}\left(\sqrt{\lambda_\chi-\underline{a}} +\sqrt{\underline{a}}\right)}{\lambda_\chi-\underline{a}}= \frac{b_{\inf}-\chi\mu}{\chi\mu}.
  $$
  This implies that for any $\lambda>\lambda_\chi$,
  $$
  \frac{\sqrt{\underline{a}}\left(\sqrt{\lambda-\underline{a}} +\sqrt{\underline{a}}\right)}{\lambda-\underline{a}}< \frac{b_{\inf}-\chi\mu}{\chi\mu},
  $$
  and hence $\sqrt{\underline{a}}<\kappa_\chi<\sqrt\lambda$, where $\kappa_\chi$ is such that \eqref{min-kappa-eq} holds.
  It then follows that, when $\lambda\ge \lambda_\chi(>\underbar{a})$, $c_\chi^*=c_0^*=2\sqrt{\underline{a}}$.
  Hence,  Theorem \ref{Main-tm1}(1) and Theorem \ref{Main-tm2}  imply that the minimal least mean speed $c_0^*=2\sqrt {\underbar{a}}$ of transition front solutions to the time heterogeneous Fisher-KPP equation
  \eqref{kpp-eq} is also the minimal least mean speed of the transition front solutions to  the chemotaxis model \eqref{P0} with  $\lambda>\lambda_\chi(>\underbar{a})$.

  \item[(5)] For fixed $\lambda>\underbar{a}$, when $\chi>0$ is sufficiently small, we also have $c_\chi^*=2\sqrt{\underbar{a}}$.
  Indeed,
    for  any given $\lambda>\underline{a}$,  we have that $c^*_{\chi}=2\sqrt{\underline{a}}$ whenever $\chi\mu<\min\Big\{
\frac{a_{\inf}b_{\inf}}{a_{\inf}+ a_{\sup} },
\frac{b_{\inf}(\lambda-\underbar{a})}{ \lambda-\underbar{a}+ \sqrt{\underline{a}}\big(\sqrt{\lambda-\underline{a}} +\sqrt{\underline{a}}\big)}
 \Big\}$. Again, Theorem \ref{Main-tm1}(1) and Theorem \ref{Main-tm2} imply that the minimal least mean speed $c_0^*=2\sqrt {\underbar{a}}$ of transition front solutions to the time heterogeneous Fisher-KPP equation
  \eqref{kpp-eq} is also the minimal least mean speed of the transition front solutions to the chemotaxis model \eqref{P0} with  $\chi$ sufficiently small.

\item[(6)] It remains open whether for fixed $0<\lambda\le \underbar{a}$, when $\chi>0$ is sufficiently small, for any $\underbar{c}>c_0^*=2\sqrt{\underbar{a}}$, \eqref{P0} has
  a transition front solution connecting $(0,0)$ and $(u^*(t),v^*(t))$ with least mean speed $\underbar{c}$.

\end{itemize}
\end{rk}

The rest of the paper is organized as follows. In section 2,  we construct proper sub-solutions and super-solutions of some equations related to \eqref{P0-1} with certain $c(t)$,
 which will be  of great use in the proofs of the main results. We prove Theorem \ref{Main-tm1} and Theorem \ref{Main-tm2} in sections 3
 and 4, respectively.

\section{Sub- and super-solutions}

In this section, we construct proper sub-solutions and super-solutions of some equations related to \eqref{P0-1} with certain $c(t)$.

For any fixed $0<\kappa<\min\{\kappa_\chi, \sqrt{\underline{a}}\} (\le \min\{\sqrt \lambda, \sqrt{\underline{a}}\})$,
let $c_{\kappa}(t)=\frac{a(t)+\kappa^2}{\kappa}$,
\begin{equation}
\label{phi-k-eq}
\phi_{\kappa}(x)=e^{-\kappa x}, \quad \forall \, x\in\R,
\end{equation}
and
\begin{equation}\label{def-of-sup-sol}
\phi_{\kappa}^{+}(x)=\min\{\phi_{\kappa}(x), \frac{a_{\sup}}{b_{\inf}-\chi\mu}\}, \quad \forall\ x\in\R.
\end{equation}
It is not difficult to see that
\begin{equation}\label{linear eqt of phi-k}
\frac{d^2}{dx^2}\phi_{\kappa} +c_\kappa(t)\frac{d}{dx}\phi_{\kappa}+a(t)\phi_{\kappa}=\Big(\kappa^2-\kappa c(t) +a(t)\Big)\phi_{\kappa}=0, \quad \forall\ x\in\R,\ t\in\R.
\end{equation}
Note that
\begin{equation}\label{c_k-hat-formula-1}
\underline{c}_{\kappa}:=\liminf_{t-s\to\infty}\frac{1}{t-s}\int_{s}^{t}c_{\kappa}(y)dy=\frac{\underline{a}+\kappa^2}{\kappa},\quad  \forall\ \kappa>0,
\end{equation}
and
\begin{equation}\label{c_k-hat-formula-2}
\overline{c}_{\kappa}:=\limsup_{t-s\to\infty}\frac{1}{t-s}\int_{s}^{t}c_{\kappa}(y)dy=\frac{\overline{a}+\kappa^2}{\kappa},\quad  \forall\ \kappa>0.
\end{equation}
Hence,
\begin{equation*}
\frac{a_{\inf}+\kappa^2}{\kappa}\leq \underline{c}_{\kappa}\le \overline{c}_\kappa\leq \frac{a_{\sup}+\kappa^2}{\kappa},\quad  \forall\ \kappa>0.
\end{equation*}

Fix $0<\kappa<\min\{\kappa_\chi, \sqrt{\underline{a}}\}$ and $0<\beta_0<\frac{1}{3}$. Let
\begin{align}
\label{def-of-set-e}
\mathcal{E}_{\kappa,\beta_0,M}=\{ \phi\in C(\R, C_{\rm unif}^b(\R))\,|& \,  0\leq \phi(x,t)\leq \phi^+_{\kappa}(x)\,\, {\rm and} |\phi(x+h,t)-\phi(x,t)|\le M |h|^{\beta_0},\nonumber\\
& |\phi(x,t+h)-\phi(x,t)|\le M |h|^{\beta_0} \,\, \forall\, x,t,h\in\R,\, |h|\le 1\},
\end{align}
where $M$ is a positive constant to be determined later. For  every $\phi\in\mathcal{E}_{\kappa,\beta_0,M}$, consider
\begin{equation}
\label{L-operator}
\mathcal{L}_{\kappa,\phi}(u)=0,
\end{equation}
where
\begin{equation*}
\mathcal{L}_{\kappa,\phi}(u)=\partial_t u-\partial_{xx}u -(c_{\kappa}(t)-\chi\partial_{x}\psi(\cdot,\cdot;\phi))\partial_{x}u -(a(t)-\chi\lambda \psi(\cdot,\cdot;\phi)-(b(t)-\chi\mu)u)u,
\end{equation*}
and $\psi(\cdot,\cdot;\phi)$ is given by
\begin{align}\label{v-formula}
\psi(x,t;\phi)&=\mu\int_{0}^{\infty}\frac{e^{-\lambda s }}{\sqrt{4\pi s}}\Big[\int_{\R}e^{-\frac{|x-y|^2}{4s}}\phi(y,t)dy \Big]ds\nonumber\\
&=\frac{\mu}{\sqrt{\pi}}\int_{0}^{\infty}\int_{\R}e^{-\lambda s-|y|^2}\phi(x+2\sqrt{s}y,t)dyds.
\end{align}
It is not difficult to prove that $\psi(\cdot,\cdot;\phi)$ solves
$$
\partial_{xx}\psi(x,t;\phi)-\lambda \psi(x,t;\phi)+\mu \psi(x,t;\phi)=0, \quad\ x\in\R,\ t\in\R.
$$

Note that, for every $\phi\in\mathcal{E}_{\kappa,\beta_0,M}$, it holds that
\begin{align}\label{space-derivatie of v}
\partial_{x}\psi(x,t;\phi)&=\mu\int_{0}^{\infty}\frac{e^{-\lambda s}}{\sqrt{4\pi s}}\Big[\int_{\R}\frac{(y-x)e^{-\frac{|x-y|^2}{4s}}}{2s}\phi(y,t) dy\Big]ds\nonumber\\
&=\frac{\mu}{\sqrt{\pi}}\int_{0}^{\infty}\frac{e^{-\lambda s}}{\sqrt{s}}\Big[\int_{\R}ze^{-z^2}\phi(x+2z\sqrt{s},t) dz\Big]ds.
\end{align}
Note also that,  for given $\phi \in\mathcal{E}_{\kappa,\beta_0,M}$, if $u(x,t)=\phi(x,t)$ is an entire solution of \eqref{L-operator}, then
$(U(x,t)$, $V(x,t))=(\phi(x,t),\psi(x,t;\phi))$ is an entire solution of \eqref{P0-1} with $c(t)=c_\kappa(t)$. In the following,
for given $\phi \in\mathcal{E}_{\kappa,\beta_0,M}$, we
construct proper sub- and super-solutions of \eqref{L-operator}.

\begin{defin}\label{sup-sub sol} For each given $\phi\in\mathcal{E}_{\kappa,\beta_0,M}$, a function $u\in C^{2,1}(D)$, where $D\subset \R\times\R$, is called a super-solution (resp.   sub-solution) of \eqref{L-operator} on $D$ if
$$
\mathcal{L}_{\kappa,\phi}(u)(x,t)\geq 0  \,\,\, (\text{resp}. \,\, \mathcal{L}_{\kappa,\phi}(u)(x,t)\leq 0)\,\,\, \text{for } \ (x,t)\in D.
$$
\end{defin}

The following Lemma provides some useful estimates on $\psi(\cdot,\cdot,\phi)$ and $\p_x\psi(\cdot,\cdot,\phi)$ for each $\phi\in\mathcal{E}_{\kappa}$.

\begin{lem}\label{estimate on v equations}
\begin{itemize}
\item[(i)] For every $\phi\in\mathcal{E}_{\kappa,\beta_0,M},$ we have that
\begin{equation}\label{pointwise-estimate on v}
0\leq \psi(x,t;\phi)\leq \min\{\frac{\mu a_{\sup}}{\lambda(b_{\inf}-\chi\mu)},\ \frac{\mu}{\lambda-\kappa^2}\phi_{\kappa}(x)\}, \quad \forall \ x\in\R,\ \forall\ t\in\R
\end{equation}
and
\begin{equation}\label{pointwise-estimate on space derivative of v}
|\partial_{x}\psi(x,t;\phi)|\leq \frac{\mu(\sqrt{\lambda-\kappa^2}+\kappa)}{\lambda-\kappa^2}\phi_{\kappa}(x),  \quad \forall \ x\in\R,\ \forall\ t\in\R.
\end{equation}

\item[(ii)] Let $\phi\in\mathcal{E}_{\kappa,\beta_0,M}$ and $\{\phi_{n}\}_{n\geq 1}\subset \mathcal{E}_{\kappa,\beta_0,M}$ such that $\phi_{n}(x,t)\to \phi(x,t)$ as $n\to\infty$ uniformly on every compact subset of $\R\times\R$. Then
\begin{equation*}
\psi(x,t;\phi_n)\to \psi(x,t;\phi) \quad\text{and}\quad \psi_x(x,t;\phi_n)\to \psi_x(x,t;\phi)\ \text{as}\ n\to\infty
\end{equation*}
uniformly on every compact subset of $\R\times\R$.
\end{itemize}
\end{lem}

\begin{proof} (i)
The following arguments are inspired from the proofs of \cite[Lemmas 2.2 \& 2.3]{SaSh2}. So we refer the reader to \cite{SaSh2} for more details on the estimates. Let $\phi\in\mathcal{E}_{\kappa,\beta_0,M}$ be given. Since $0\leq \phi(x,t)\leq \phi_{\kappa}(x)$, it follows from \eqref{v-formula} that
\begin{equation}\label{a-1}
\begin{split}
\psi(x,t;\phi) \leq \mu\int_{0}^{\infty}\frac{e^{-\lambda s }}{\sqrt{4\pi s}}\Big[\int_{\R}e^{-\frac{|x-y|^2}{4s}}e^{-\kappa y}dy \Big]ds=\frac{\mu}{\lambda-\kappa^2}\phi_{\kappa}(x).
\end{split}
\end{equation}
On the other hand, we have
\begin{equation}\label{a-2}
\begin{split}
0\leq \psi(x,t;\phi)&\leq \mu\Big[\int_{0}^{\infty}\frac{e^{-\lambda s }}{\sqrt{4\pi s}}\Big[\int_{\R}e^{-\frac{|x-y|^2}{4s}}dy \Big]ds\Big]\sup_{x\in\R}\phi(x,t)=\frac{\mu}{\lambda}\sup_{x\in\R}\phi(x,t).
\end{split}
\end{equation}
Inequality \eqref{pointwise-estimate on v} follows from \eqref{a-1} and \eqref{a-2}.

Using \eqref{space-derivatie of v} we have
\begin{equation}\label{a-3}
\begin{split}
|\partial_{x}\psi(x,t;\phi)|\leq \frac{\mu e^{-\kappa x}}{\sqrt{\pi}}\int_{0}^{\infty}\frac{e^{-(\lambda-\kappa^2) s}}{\sqrt{s}}\Big[\int_{\R}(|z| +\kappa
\sqrt{s})e^{-z^2}dz\Big]ds=\frac{\mu(\sqrt{\lambda-\kappa^2}+\kappa)}{\lambda-\kappa^2}\phi_{\kappa}(x).
\end{split}
\end{equation}
Inequality \eqref{pointwise-estimate on space derivative of v} then follows from \eqref{a-3}. This completes the proof of (i).

(ii) Let $\phi\in\mathcal{E}_{\kappa,\beta_0,M}$ and $\{\phi_{n}\}_{n\geq 1}\subset \mathcal{E}_{\kappa,\beta_0,M}$ be such that $\phi_{n}(x,t)\to \phi(x,t)$ as $n\to\infty$ uniformly on every compact subset of $\R\times\R$. Then
$$
\lim_{n\to\infty}\sup_{ (x,t)\in[-K,K]^2}\int_0^R\int_{-R}^{R}
e^{-\lambda s}e^{-|y|^2}|\phi_{n}(x+2\sqrt{s}y,t)-\phi(x+2\sqrt{s}y,t)|dyds=0, \forall K>0, \ R>0.
$$
On the other hand, observe that
$$
\lim_{R\to\infty}\sup_{n\geq 1, x\in\R,\ t\in\R}\int_{\{s\geq R \ \text{or}\ |y|\geq R\}}e^{-\lambda s}e^{-|y|^2}\phi_{n}(x+2\sqrt{s}y,t)dyds=0.
$$
Therefore, it follows from \eqref{v-formula} that $\psi(x,t;\phi_n)\to \psi(x,t;\phi)$ as $n\to\infty$ uniformly on every compact subset of $\R\times\R$.

Similarly, using \eqref{space-derivatie of v}, the similar arguments to the above yield that  $\psi_x(x,t;\phi_n)\to \psi_x(x,t;\phi)$ as $n\to\infty$ uniformly on every compact subset of $\R\times\R$.
\end{proof}

 By \eqref{c_k-hat-formula-1}, we have that
\begin{equation}\label{lower bound for c(t)}
\underline{c}_{\kappa}-2\kappa\geq \frac{\underline{a}-\kappa^2}{\kappa}>0, \quad \forall\ 0<\kappa< \min\{\kappa_\chi,\sqrt{\underline{a}}\}.
\end{equation}
Hence, since  {\bf (H1)} holds, there is $\varepsilon>0$ such that
\begin{equation}\label{espilon choice}
0<\varepsilon<\min\{\kappa, \frac{\underline{a}-\kappa^2}{\kappa}\} \quad\text{and}\quad \frac{b_{\inf}-\chi\mu}{\chi\mu}> \frac{(\kappa+\varepsilon)(\sqrt{\lambda-\kappa^2}+\kappa)-\lambda}{\lambda-\kappa^2}.
\end{equation}
With this choice of $\varepsilon$, it readily follows that
$\underline{c}_{\kappa}-2\kappa-\varepsilon=\frac{1}{\kappa}(\underline{a}-\kappa(\kappa+\varepsilon))>0.$

\begin{lem}
\label{A-lem}
Fix an $\epsilon>0$ satisfying \eqref{espilon choice}.
 There is $A\in W_{loc}^{1,\infty}(\R)\cap L^{\infty}(\R)$ such that
\begin{equation}\label{A0 eq}
A'(t)+\varepsilon(c_{\kappa}(t)-2\kappa-\varepsilon)>A_0:=\frac{\varepsilon}{2}(\underline{c}_{\kappa}-2\kappa-\varepsilon)=\frac{\varepsilon}{2\kappa}(\underline{a}-\kappa(\kappa+\varepsilon))>0, \quad \text{ a.\, e.\,\,   }\ t\in\R.
\end{equation}
Moreover,  there exist $\{t_{n}\}_{n\in\mathbb{Z}}$ such that $t_{n}<t_{n+1}$, $t_n\to\pm\infty$ as $n\to\pm\infty$
 and $A\in C^1(t_n,t_{n+1})$ for every $n\in\mathbb{Z}$.
\end{lem}

\begin{proof}
 Note that $0<\underline{c}_{\kappa}\leq \overline{c}_{\kappa}<\infty$. The lemma follows from  \cite[Lemma 2.2]{SaSh_7_I} and its proof.
\end{proof}

We introduce the following expressions
\begin{equation}\label{A1 eq}
A_1=b_{\sup}-\chi\mu+\frac{\chi\mu\Big(\kappa(\sqrt{\lambda^2-\kappa^2}+\kappa)+\lambda\Big)}{\lambda-\kappa^2}
\end{equation}
and
\begin{equation}\label{A2 eq}
A_2=b_{\inf}-\chi\mu+\frac{\chi\mu\Big(\lambda-(\kappa+\varepsilon)(\sqrt{\lambda-\kappa^2}+\kappa)\Big)}{\lambda-\kappa^2}.
\end{equation}
It follows  from \eqref{espilon choice} that $A_1>0$ and $A_2>0$. Finally, let us take
\begin{equation}\label{choice-of-d-eq}
d=e^{\|A\|_{\infty}}(1+\frac{A_0}{A_1})
\end{equation} and define
\begin{equation}\label{phi lower}
\phi_{\kappa,\varepsilon,d,A}(x,t)=\phi_{\kappa}(x)-de^{A(t)}\phi_{\kappa+\varepsilon}(x).
\end{equation}
We introduce the following functions
\begin{equation}\label{D-boundary equation}
x_{\kappa,\varepsilon,d,A}^-(t):=\frac{\ln(d)+A(t)}{\varepsilon}=\frac{\|A\|_{\infty}+A(t)+\ln(1+\frac{A_0}{A_1})}{\varepsilon}>0,
\end{equation}
and
\begin{equation}\label{x plus def}
x^{+}_{\kappa,\varepsilon,d,A}(t):=\frac{\ln(\frac{\kappa+\varepsilon}{\kappa}d)+A(t)}{\varepsilon}, \quad \forall\, t\in\R.
\end{equation}
It is clear from the definition of $x_{\kappa,\varepsilon,d,A}^{+}(t)$ and $x_{\kappa,\varepsilon,d,A}^{-}(t) $ that
\begin{equation}\label{z-eq1}
x_{\kappa,\varepsilon,d,A}^{+}(t)-x_{\kappa,\varepsilon,d,A}^{-}(t)=\frac{\ln(\frac{\kappa+\varepsilon}{\kappa})}{\varepsilon},\quad \forall\, t\in\R.
\end{equation}
The next result provides some useful information on the relationship between the functions $\phi_{\kappa,\varepsilon,d,A}(x,t)$, $x^{+}_{\kappa,\varepsilon,d,A}(t)$, and $x^{-}_{\kappa,\varepsilon,d,A}(t)$.

\begin{lem}\label{Lem 2} Let  $\phi_{\kappa,\varepsilon,d,A}(x,t)$, $x^{+}_{\kappa,\varepsilon,d,A}(t)$, and $x^{-}_{\kappa,\varepsilon,d,A}(t)$ be given by \eqref{phi lower},  \eqref{x plus def}, and  \eqref{D-boundary equation},  respectively. Then, the following hold.
\begin{description}
\item[(i)] For every $t\in\R$, we have that
\begin{equation}\label{x minus def}
\phi_{\kappa,\varepsilon,d,A}(x_{\kappa,\varepsilon,d,A}^-(t),t)=0 \,\, \text{and}\,\, (x-x_{\kappa,\varepsilon,d,A}^-(t)) \phi_{\kappa,\varepsilon,d,A}(x,t)>0, \ \forall x\not = x_{\kappa,\varepsilon,d,A}^-(t).
\end{equation}
\item[(ii)]  For every $t\in\R$, we have that
\begin{equation}
0<\frac{\ln(d)-\|A\|_{\infty}}{\varepsilon}\leq x_{\kappa,\varepsilon,d,A}^-(t) < x_{\kappa,\varepsilon,d,A}^+(t)\leq \frac{\ln(\frac{\kappa+\varepsilon}{\kappa}d)+\|A\|_{\infty}}{\varepsilon} <\infty.
\end{equation}
\item[(iii)] For every $t\in\R$, the function $\R\ni x\mapsto \phi_{\kappa,\varepsilon,d,A}(x,t)$ is strictly increasing on the interval $(-\infty,  x_{\kappa,\varepsilon,d,A}^+(t)]$ and is strictly decreasing on the interval $[ x_{\kappa,\varepsilon,d,A}^+(t),\infty)$. Hence, we have that
$\max_{x\in\R}\phi_{\kappa,\varepsilon,d,A}(x,t)= \phi_{\kappa,\varepsilon,d,A}( x_{\kappa,\varepsilon,d,A}^+(t),t)$ for each $t\in\R$. Moreover, it holds that
\begin{equation}
0< \inf_{t\in\R}\phi_{\kappa,\varepsilon,d,A}( x_{\kappa,\varepsilon,d,A}^+(t),t).
\end{equation}
\end{description}
\end{lem}
\begin{proof}
Lemma \ref{Lem 2} (i) and (ii) follow from \eqref{D-boundary equation}, \eqref{x plus def} and \eqref{z-eq1}. Observe that
$$
\partial_{x}\phi_{\kappa,\varepsilon,d,A}( x,t)=-\kappa \Big(e^{\varepsilon x}-\frac{(\kappa+\varepsilon)de^{A(t)}}{\kappa} \Big)e^{-(\kappa+\varepsilon)x}, \quad \forall\ x\in\R,\ \forall\ t\in\R,
$$
and $$\partial_{x}\phi_{\kappa,\varepsilon,d,A}( x_{\kappa,\varepsilon,d,A}^+(t),t)=0, \quad \forall\ t\in\R.$$
Hence the first statement in Lemma \ref{Lem 2} (iii) follows from the first derivative test.  We have that
$$
\phi_{\kappa,\varepsilon,d,A}( x_{\kappa,\varepsilon,d,A}^+(t),t)=\frac{\varepsilon}{\kappa}e^{A(t)-(\kappa+\varepsilon)x^+_{\kappa,\varepsilon, d,A}(t)}\geq \frac{\varepsilon}{\kappa}e^{-\frac{(\kappa+\varepsilon)}{\varepsilon}\Big(2\|A\|_{\infty}+\ln(d\frac{\kappa+\varepsilon}{\kappa})\Big)}>0,\quad \forall\, t\in\R.
$$
This completes the proof of the Lemma.
\end{proof}

The next result provides us with sub- and super-solutions of \eqref{L-operator} for every $\phi\in\mathcal{E}_{\kappa,\beta_0,M}$.
\begin{tm}\label{strong-sup-solution}
Suppose that {\bf (H1)} holds and let  $\phi\in\mathcal{E}_{\kappa,\beta_0,M}$. Then the following hold.
\begin{description}
\item[(i)] The function $u(x,t)=\phi_{\kappa}(x)$ is a super-solution of \eqref{L-operator} on $\R\times\R$.

\item[(ii)] The constant function $u(x,t)=\frac{a_{\sup}}{b_{\inf}-\chi\mu}$ is a super-solution of  \eqref{L-operator} on $\R\times\R$.

\item[(iii)]  The function $\phi_{\kappa,\varepsilon,d,A}(x,t)$ is a sub-solution of \eqref{L-operator} on the set $D_{\kappa,\varepsilon,d,A}$ defined by
$$D_{\kappa,\varepsilon,d,A}:=\{(x,t)\in\R\times\R\ :\ x\geq x_{\kappa,\varepsilon,d,A}^-(t)\}.$$

\item[(iv)] For every $0<\delta\leq \frac{b_{\inf}-\chi\mu\Big(1+\frac{a_{\sup}}{a_{\inf}}\Big)}{(b_{\inf}-\chi\mu)(b_{\sup}-\chi\mu)}$, the constant function $u(x,t)=\delta$ is a sub-solution of  \eqref{L-operator} on $\R\times\R$.
\end{description}

\end{tm}

\begin{proof} Let $\phi\in\mathcal{E}_{\kappa,\beta_0,M}$ be given.  The theorem can be proved by the similar arguments as those in \cite[Theorem 2.1]{SaSh2}. For the completeness, we provide a proof in the following.

(i) Using \eqref{linear eqt of phi-k} and \eqref{pointwise-estimate on space derivative of v}, for every $(x,t)\in\R\times\R$, we have
\begin{equation}\label{a-5}
\begin{split}
\mathcal{L}_{\kappa,\phi}(\phi_{\kappa})(x,t)&=-\kappa\chi\partial_{x}\psi(x,t;u)\phi_{\kappa}+(\chi\lambda \psi(x,t;\phi)+(b(t)-\chi\mu)\phi_{\kappa})\phi_{\kappa}\\
&\geq \Big(b_{\inf}-\chi\mu-\kappa\chi\frac{\mu(\sqrt{\lambda-\kappa^2}+\kappa)}{\lambda-\kappa^2}\Big)\phi_{\kappa}^2+\chi\lambda v(\cdot,\cdot;u)\phi_{\kappa}\geq 0.
\end{split}
\end{equation}
Hence (i) follows.

(ii) Since $0\leq \psi(x,t;\phi)$,  we have that
\begin{equation*}
\begin{split}
\mathcal{L}_{\kappa,\phi}(\frac{a_{\sup}}{b_{\inf}-\chi\mu})=&-\Big(a(t)-\chi\lambda \psi(\cdot,\cdot;\phi)-(b(t)-\chi\mu)\frac{a_{\sup}}{b_{\inf}-\chi\mu} \Big)\frac{a_{\sup}}{b_{\inf}-\chi\mu}\\
\geq & \Big((b(t)-\chi\mu)\frac{a_{\sup}}{b_{\inf}-\chi\mu} -a(t)\Big)\frac{a_{\sup}}{b_{\inf}-\chi\mu}\ge 0.
\end{split}
\end{equation*}
Hence (ii) follows.

(iii) We first note that $x>0$ whenever  $(x,t)\in  D_{\kappa,\varepsilon,d,A}$. For   $(x,t)\in  D_{\kappa,\varepsilon,d,A}$,   
let $\tilde A(t,x)=A(t)-(\kappa+\varepsilon)x$ and using Lemma \ref{estimate on v equations} and  Lemma \ref{A-lem}, we have
\begin{equation*}
 \begin{split}
 &\mathcal{L}_{\kappa,\phi}(\phi_{\kappa,\varepsilon,d,A})(x,t)\\
 =& -dA'e^{\tilde A(t,x)}-\Big[\kappa^2 e^{-\kappa x}-d(\kappa+\varepsilon)^2e^{\tilde A(t,x) }\Big]-(c_{\kappa}(t)-\chi\p_{x}\psi(\cdot,\cdot;\phi))\Big[-\kappa e^{-\kappa x}+d(\kappa+\varepsilon)e^{\tilde A(t,x)} \Big]\\
 &-\Big[ a(t)-\chi\lambda\psi(\cdot,\cdot;\phi)-(b(t)-\chi\mu)\phi_{\kappa,\varepsilon,d,A} \Big]\phi_{\kappa,\varepsilon,d,A}\\
 =& -d\Big(A'+\frac{\varepsilon}{\kappa}\Big(a(t)-\kappa(\kappa+\varepsilon) \Big) \Big)e^{\tilde A(t,x)}+\chi\partial_{x}\psi(\cdot,\cdot;\phi)\Big(-\kappa e^{-\kappa x} +d(\kappa+\varepsilon) e^{\tilde A(t,x)}\Big)\\
 &+\chi\lambda\psi(\cdot,\cdot;\phi)\phi_{\kappa,\varepsilon,d,A} +(b(t)-\chi\mu)e^{-2\kappa x} -d(b(t)-\chi\mu)\Big(e^{-\kappa x}+\phi_{\kappa,\varepsilon,d,A} \Big)e^{\tilde A(t,x)}\\
 \leq & -d\Big(A'+\varepsilon(c_{\kappa}(t)-2\kappa-\varepsilon)\Big)e^{\tilde A(t,x)}+\frac{\chi\lambda\mu}{\lambda-\kappa^2}\phi_{\kappa}(\phi_{\kappa}-de^{\tilde A(t,x)})\\
 & +\frac{\chi\mu(\sqrt{\lambda-\kappa^2}+\kappa)}{\lambda-\kappa^2}\phi_{\kappa}\Big(\kappa e^{-\kappa x}  +d(\kappa+\varepsilon)e^{\tilde A(t,x)} \Big)\\
 & +(b(t)-\chi\mu)e^{-2\kappa x} -d(b(t)-\chi\mu)\Big(\phi_{\kappa}+\phi_{\kappa,\varepsilon,d,A} \Big)e^{\tilde A(t,x)}\\
 \leq & -\Big(\frac{d\varepsilon}{2\kappa}\Big(\underline{a}-\kappa(\kappa+\varepsilon) \Big)e^{A(t)+(\kappa-\varepsilon)x}- (b(t)-\chi\mu)-\frac{\chi\mu(\lambda +\kappa(\sqrt{\lambda-\kappa^2}+\kappa))}{\lambda-\kappa^2}\Big)e^{-2\kappa x} \\
 &-d\Big((b(t)-\chi\mu)+  \chi\mu\Big(\frac{\lambda-(\kappa+\varepsilon)(\sqrt{\lambda-\kappa^2}+\kappa)}{\lambda-\kappa^2} \Big)\Big)\phi_{\kappa}e^{\tilde A(t,x)} -d(b(t)-\chi\mu)\phi_{\kappa,\varepsilon,d,A}e^{\tilde A(t,x)}\\
 \leq& - A_{0}e^{-\|A\|_{\infty}}\Big(d- \frac{A_{1}e^{\|A\|_{\infty}}}{A_{0}} \Big)e^{-2\kappa x}-dA_2\phi_{\kappa}e^{\tilde A(t,x)}-d(b_{\inf}-\chi\mu)\phi_{\kappa,\varepsilon,d,A}e^{\tilde A(t,x)}\\
 \leq &0,
 \end{split}
 \end{equation*}
 where $A_0$, $A_1$, and $A_2$ are given by \eqref{A0 eq}, \eqref{A1 eq} and \eqref{A2 eq} respectively. Thus, (iii) follows.

 (iv) Let  $0<\delta\leq \frac{b_{\inf}-\chi\mu\Big(1+\frac{a_{\sup}}{a_{\inf}}\Big)}{(b_{\inf}-\chi\mu)(b_{\sup}-\chi\mu)}$. We have that
 \begin{equation*}
\begin{split}
\mathcal{L}_{\kappa,\phi}(\delta)=&-\Big(a(t)-\chi\lambda \psi(\cdot,\cdot;\phi)-(b(t)-\chi\mu)\delta \Big)\delta\\
\leq & -(a_{\inf}-\frac{\chi\mu a_{\sup}}{b_{\inf}-\chi\mu}-(b_{\sup}-\chi\mu)\delta)\delta\\
=&-a_{\inf}(b_{\sup}-\chi\mu)\Big( \frac{b_{\inf}-\chi\mu\big(1+\frac{a_{\sup}}{a_{\inf}}\big)}{(b_{\inf}-\chi\mu)(b_{\sup}-\chi\mu)}-\delta\Big)\\
\leq &0.
\end{split}
\end{equation*}
 (iv) then follows.
\end{proof}

We recall from Lemma \ref{Lem 2} (iv) that $ \inf_{t\in\R}\phi_{\kappa,\varepsilon,d,A}(x^{+}_{\kappa,\varepsilon,d,A}(t),t) >0$, where $x^{+}_{\kappa,\varepsilon,d,A}(t) $ is defined by \eqref{x plus def}. Moreover, Lemma \ref{Lem 2} guarantees that for each given $t\in\R$ and every $\delta\in (0 , \inf_{t\in\R}\phi_{\kappa,\varepsilon,d,A}( x^{+}_{\kappa,\varepsilon,d,A}(t) ))$, there is a unique $x_{\kappa,\varepsilon,d,A}(t;\delta)\in (x^{-}_{\kappa,\varepsilon,d,A}(t) , x^{+}_{\kappa,\varepsilon,d,A}(t) )$ such that
$$
\phi_{\kappa,\varepsilon,d,A}(x_{\kappa,\varepsilon,d,A}(t;\delta),t) =\delta.
$$
 Let $0<\delta_0<\min\{\frac{b_{\inf}-\chi\mu\Big(1+\frac{a_{\sup}}{a_{\inf}}\Big)}{(b_{\inf}-\chi\mu)(b_{\sup}-\chi\mu)}, \inf_{t\in\R}\phi_{\kappa,\varepsilon,d,A}(x^{+}_{\kappa,\varepsilon,d,A}(t),t)\}$ be fixed and define
 \begin{equation}\label{time dependent lower - solution}
 \phi_{\kappa}^{-}(x,t)=
 \begin{cases}
 \phi_{\kappa,\varepsilon,d,A}(x,t), \quad \text{if}\  x\geq x_{\kappa,\varepsilon,d,A}(t;\delta_0)\cr
 \delta_0, \qquad \qquad \quad \ \text{if}\  x\leq x_{\kappa,\varepsilon,d,A}(t;\delta_0).
 \end{cases}
\end{equation}
It should be noted that the function $\R\ni t\mapsto x_{\kappa,\varepsilon,d,A}(t;\delta_0)$ is Lipschitz  continuous.

\section{Existence of transition front solutions}

In this section, we prove the existence of transition front solutions of \eqref{P0}. We suppose that {\bf (H1)} holds, and $\mathcal{E}_{\kappa,\beta_0,M}$ is defined as in the previous section. The functions $\phi_{\kappa}^{+}$ and $\phi_{\kappa}^{-}$ are  given by \eqref{def-of-sup-sol} and \eqref{time dependent lower - solution}, respectively. {Our main idea to prove the existence of transition front solutions of \eqref{P0}
is to prove that there is $\phi(\cdot,\cdot)\in  \mathcal{E}_{\kappa,\beta_0,M}$ such that $(U(x,t),V(x,t))=(\phi(x,t),\psi(x,t;\phi))$ is an entire solution of
\eqref{P0-1} with $c(t)=c_\kappa(t)$ and that $(u(x,t),v(x,t))=(U(x-\int_0^ tc_\kappa(s)ds,t), V(x-\int_0^t c_\kappa(s)ds,t))$ is a transition front solution of \eqref{P0}.
To do so, we first prove some lemmas.}

\smallskip

For every  $\phi\in \mathcal{E}_{\kappa,\beta_0,M}$, $t_0\in\R$ and $u_0\in C^{b}_{\rm unif}(\R)$,  let $\Phi(x,t;t_0,u_0,\phi)$, $x\in\R$, $t\geq t_0$ be the solution of
\begin{equation}\label{Un def}
\begin{cases}
\mathcal{L}_{\kappa,\phi}(\Phi)=0, \quad x\in\R, \ t>t_0\\
\Phi(x,t_0;t_0,u_0,\phi)=u_0(x), \quad \forall\ x\in\R,
\end{cases}
\end{equation}
where $\mathcal{L}_{\kappa,\phi}$ is given by \eqref{L-operator} (the existence of $\Phi(x,t;t_0,u_0,\phi)$ follows from general semigroup theory).

\begin{lem}\label{Main lem 3}
For every $\phi\in\mathcal{E}_{\kappa,\beta_0,M}$ and $t_0\in\R$, let $\Phi(x,t;t_0,\phi_{\kappa}^+,\phi)$ be given by \eqref{Un def}. Then the following hold.
\begin{description}
\item[(i)] For any $t_2<t_1$, $0\leq \Phi(x,t;t_2,\phi_{\kappa}^+,\phi)\leq \Phi(x,t;t_1,\phi_{\kappa}^+,\phi)\leq \phi_{\kappa}^+(x)$ for every $x\in\R$ and $ t\geq t_1$.

\item[(ii)] For any $t_0\in\R$, $\phi_{\kappa}^{-}(x,t)\leq \Phi(x,t;t_0,\phi_{\kappa}^+,\phi)$ for all $x\in\R$ and $t\geq t_0$, where $\phi_{\kappa,\varepsilon,d,A}$ is given by \eqref{phi lower}.
\end{description}

\end{lem}
\begin{proof} (i) Since $\Phi(x,t_0;t_0,\phi_{\kappa}^+,\phi)\leq \phi_{\kappa}(x)$ for every $x\in\R$, it follows from comparison principle for parabolic equations and Theorem \ref{strong-sup-solution} (i) that
$$
\Phi(x,t;t_0,\phi_{\kappa}^+,\phi)\leq \phi_{\kappa}(x), \quad  \forall \ x\in\R,\ \forall\ t\geq t_0.
$$
On the other hand, since $\phi_{\kappa}^+(x)\leq\frac{a_{\sup}}{b_{\inf}-\chi\mu}$, it follows from Theorem \ref{strong-sup-solution} (ii) and  comparison principle for parabolic equations that
$$
\Phi(x,t;t_0,\phi_{\kappa}^+,\phi)\leq \frac{a_{\sup}}{b_{\inf}-\chi\mu}, \quad  \forall \ x\in\R,\ \forall\ t\geq t_0.
$$
Thus
$$
\Phi(x,t;t_0,\phi_{\kappa}^+,\phi)\leq \phi_{\kappa}^+(x)=\Phi(x,t_0;t_0,\phi_{\kappa}^+,\phi) \quad  \forall \ x\in\R,\ \forall\ t\geq t_0.
$$
Hence for any $t_2<t_1$,
$$
\Phi(x,t_1;t_2,\phi_{\kappa}^+,\phi)\leq\phi_{\kappa}^+(x)= \Phi(x,t_1;t_1,\phi_{\kappa}^+,\phi).
 $$
 Then, by comparison principle for parabolic equations, we have
$$
\Phi(x,t;t_2,\phi_{\kappa}^+,\phi)\leq \Phi(x,t;t_1,\phi_{\kappa}^+,\phi) \quad  \forall \ x\in\R,\ \forall\ t\geq t_1.
$$
This completes the proof of (i).

(ii) Note that if $\phi_{\kappa,\varepsilon,d,A}(x,t)\leq 0$ then we are done.
 Observe from (i) and Lemma \ref{Lem 2} (i) that $\phi_{\kappa,\varepsilon,d,A}(x_{\kappa,\varepsilon,d}^-(t),t)=0\leq \Phi(x_{\kappa,\varepsilon,d,A}^-(t),t;t_0,\phi_{\kappa}^{+},\phi)$ for every $t\geq t_0$.  Since $\phi_{\kappa,\varepsilon,d,A}(x,t_0)\leq \phi_{\kappa}^+(x)$ for every $x\geq x_{\kappa,\varepsilon,d,A}^-(t_0)$, hence it  follows from Theorem \ref{strong-sup-solution} (iii) and comparison principle for parabolic equations that
$$
\phi_{\kappa,\varepsilon,d,A}(x,t)\leq \Phi(x,t;t_0,\phi_{\kappa}^+,\phi),\quad \forall \ x\geq x^-_{\kappa,\varepsilon,d,A}(t),\quad  \forall t\geq t_0.
$$
 In particular, we have that
 $$
\delta_0=\phi_{\kappa,\varepsilon,d,A}(x_{\kappa,\varepsilon,d,A}(t;\delta_0),t)\leq \Phi(x_{\kappa,\varepsilon,d,A}(t;\delta_0),t;t_0,\phi_{\kappa}^+,\phi),\quad   \forall t\geq t_0.
$$
 Observe that $\delta_0\leq \phi_{\kappa}^{+}(x)=\phi(x,t_0;t_0,\phi_{\kappa}^+,\phi)$ for every $x\leq x_{\kappa,\varepsilon,d,A}(t;\delta_0)$, $t\in\R$ and $t_0\in\R$. We also note that Lemma \ref{Lem 2} implies that $u(x,t)=\delta_0$ is a sub-solution of \eqref{L-operator} on $(-\infty, x_{\kappa,\varepsilon,d,A}(\delta_0,t_0)]$. Hence, it follows from comparison principle for parabolic equations that
 $$
 \delta_0\leq \Phi(x,t;t_0,\phi_{\kappa}^+,\phi),\quad \forall\, x\leq x_{\kappa,\varepsilon,d,A}(\delta_0,t_0), \,\, \forall\, t\geq t_0.
 $$
This completes the proof of (ii).
\end{proof}

It follows from Lemma \ref{Main lem 3}  that for every $\phi\in\mathcal{E}_{\kappa,\beta_0,M}$, $\Phi(\cdot,\cdot;t_0,\phi_{\kappa}^+,\phi)$ is non-decreasing in $t_0\in\R$. Then for each $\phi\in\mathcal{E}_{\kappa,\beta_0,M}$, there is $\Phi(x,t;\phi)$ such that
\begin{equation}\label{Main-U-def}
\Phi(x,t;\phi):=\lim_{t_0\to -\infty}\Phi(x,t;t_0,\phi_{\kappa}^+,\phi)\quad \forall\, x,t\in\R.
\end{equation}
Moreover,  we have
\begin{equation}\label{z-eq2}
\phi_{\kappa}^+(x)\ge \Phi(x,t;\phi)\geq \phi_{\kappa}^-(x,t),\quad \forall\ x\in\R,\,\,\forall\ t\in\R.
\end{equation}
Hence, we introduce the following set,
\begin{equation}\label{E-kappa}
\tilde{\mathcal{E}}_{\kappa,\beta_0,M}=\{\phi\in\mathcal{E}_{\kappa,\beta_0,M}\,|\, \phi_{\kappa}^-(x,t)\leq \phi(x,t)\leq \phi_{\kappa}^+(x), \quad \forall\ x\in\R, \ \forall\ t\in\R\}.
\end{equation}

\begin{lem}\label{U-eq lem}
For every $\phi\in\tilde{\mathcal{E}}_{\kappa,\beta_0,M}$,  $\Phi(\cdot,\cdot;\phi) \in\tilde{\mathcal{E}}_{\kappa,\beta_0,M}$ when $M$ is sufficiently large and  satisfies the parabolic equation
\begin{equation}\label{Main U-eq}
\partial_{t}\Phi=\partial_{xx}\Phi+(c_{\kappa}(t)-\chi\partial_{x}\psi(\cdot,\cdot;\phi))\partial_x\Phi+(a(t)-\lambda\chi \psi(\cdot,\cdot;\phi)-(b(t)-\chi\mu)\Phi)\Phi,\,\, \forall (x,t)\in\R\times\R.
\end{equation}
\end{lem}

\begin{proof} We will apply the similar  arguments as those  in \cite[Lemma 3.5]{SaSh2} to prove this lemma.

First of all, let $\{T(t)\}_{t\ge 0}$ denotes the $C_0-$semigroup generated by $\Delta-I$ on $C^b_{\rm unif}(\R^N)$.  Using the fact that $\p_{xx}\psi(\cdot,\cdot,\phi)=\lambda \psi(\cdot,\cdot,\phi)-\mu \phi$, the variation of constant formula yields that
\begin{equation}
\begin{split}
&\Phi(x,t;t_0,\phi_{\kappa}^+,\phi)\\
=&T(t-t_0)\phi_{\kappa}^+ + \int_{t_0}^tT(t-s)((c_{\kappa}(s)-\chi\p_x \psi(\cdot,s,\phi))\p_x\Phi(\cdot,s;t_0,\phi_{\kappa}^+,\phi))ds \\
&+ \int_{t_0}^tT(t-s)((a(s)+1-\lambda\chi \psi(\cdot,s;\phi)-(b(s)-\chi\mu)\Phi(\cdot,s;t_0,\phi_{\kappa}^+,\phi))\Phi(\cdot,s;t_0,\phi_{\kappa}^+,\phi))ds\\
=& \underbrace{T(t-t_0)\phi_{\kappa}^+}_{I^1_{t_0}(\cdot,t)} + \underbrace{\int_{t_0}^tT(t-s)\partial_x((c_{\kappa}(s)-\chi\p_x \psi(\cdot,s,\phi))\Phi(\cdot,s;t_0,\phi_{\kappa}^+,\phi))ds}_{I^2_{t_0}(\cdot,t)} \\
&+ \underbrace{\int_{t_0}^tT(t-s)(a(s)+1-\chi\mu \phi(\cdot,s))\Phi(\cdot,s;t_0,\phi_{\kappa}^+,\phi)ds}_{I^3_{t_0}(\cdot,t)}\\
&-\underbrace{\int_{t_0}^tT(t-s)(b(s)-\chi\mu)\Phi^2(\cdot,s;t_0,\phi_{\kappa}^+,\phi)ds}_{I^4_{t_0}(\cdot,t)}.
\end{split}
\end{equation}

Next, choose $\frac{\beta_0}{2}<\beta<\frac{1}{2}-\beta_0$ (such $\beta$ exists for $0<\beta_0<\frac{1}{3}$). Let $X^{\beta}$ denotes the fractional power space associated with $\Delta-I$ on $C^{b}_{\rm unif}(\R^N)$.
We  claim that there is a constant $\tilde{C}_{\beta}$ independent of $t$, $u$, and $t_0$ such that
\begin{equation}\label{b-eq1}
\|\Phi(\cdot,t;t_0,\phi_{\kappa}^+,\phi)\|_{X^{\beta}}\leq \tilde{C}_{\beta}\Big(\frac{e^{-(t-t_0)}}{(t-t_0)^{\beta}}+1\Big),\quad  \forall t>t_0.
\end{equation}
In fact,
by \cite[Lemma 3.2]{SaSh1}, there is a constant $C_{\beta}>0$ such that
\begin{equation}\label{1-eq}
\|I^1_{t_0}(\cdot,t)\|_{X^\beta}\leq C_{\beta}(t-t_0)^{-\beta}e^{-(t-t_0)}\|\phi_{\kappa}^+\|_{\infty}=\frac{C_{\beta}e^{-(t-t_0)}a_{\sup}}{(t-t_0)^{\beta}(b_{\inf}-\chi\mu)},
\end{equation}
\begin{align}\label{2-eq}
\|I_{t_0}^2(\cdot,t)\|_{X^\beta}&\leq  C_{\beta}\int_{t_0}^{t}(t-s)^{-(\frac{1}{2}+\beta)}e^{-(t-s)}(\|c_{\kappa}\|_{\infty}
+\chi\|\p_{x}\psi(\cdot,s,\phi)\|_{\infty})\|\Phi(\cdot,s;t_0,\phi_{\kappa}^+,\phi)\|_{\infty}ds\nonumber\\
&\leq C_{\beta}\int_{t_0}^{t}(t-s)^{-(\frac{1}{2}+\beta)}e^{-(t-s)}(\frac{a_{\sup}+\kappa^2}{\kappa}+\chi\mu\|\phi(\cdot,s)\|_{\infty})\|\phi_{\kappa}^+\|_{\infty}ds
\nonumber\\
&\leq  C_{\beta}\Big( \frac{a_{\sup}+\kappa^2}{\kappa}+\frac{\chi\mu a_{\sup}}{b_{\inf}-\chi\mu}\Big)\frac{a_{\sup}\Gamma(\frac{1}{2}-\beta)}{b_{\inf}-\chi\mu},
\end{align}
\begin{align}\label{3-eq}
\|I^3_{t_0}(\cdot,t)\|_{X^\beta}&\leq
 C_{\beta}\int_{t_0}^t(t-s)^{-\beta}e^{-(t-s)}\Big(a_{\sup}+1+\chi\mu\|\phi(\cdot,s)\|_{\infty} \Big)\|\Phi(\cdot,s;t_0,\phi_{\kappa}^+,\phi)\|_{\infty}ds\nonumber\\
&\leq  C_{\beta}\Big( a_{\sup}+1+\frac{\chi\mu a_{\sup}}{b_{\inf}-\chi\mu}\Big)\frac{a_{\sup}\Gamma(1-\beta)}{b_{\inf}-\chi\mu},
\end{align}
and
\begin{equation}\label{4-eq}
\|I^4_{t_0}(\cdot,t)\|_{X^{\beta}}\leq C_{\beta}(b_{\sup}-\chi\mu)\frac{a_{\sup}^2\Gamma(1-\beta)}{(b_{\inf}-\chi\mu)^2}.
\end{equation}
The claim then follows from \eqref{1-eq}--\eqref{4-eq}.

Now, we claim that there is a constant $\bar{C}_{\beta}>0$ independent of $t_0$, $\phi$, and $t$ such that
\begin{equation}\label{b-eq2}
\|\Phi(\cdot,t+h;t_0,\phi_{\kappa}^+,\phi)-\Phi(\cdot,t;t_0,\phi_{\kappa}^+,\phi)\|_{X^{\beta}}\leq \bar{C}_{\beta}(h^{1-\beta}+h^{\frac{1}{2}-\beta})\Big(\frac{e^{-(t-t_0)}}{(t-t_0)^{\beta}}+1 \Big),  \ \forall h>0,\ t>t_0.
\end{equation}
In fact,  for $h>0$ and $t>t_0$, we have
\begin{equation}\label{1'-eq}
\|I_{t_0}^1(\cdot,t+h)-I_{t_0}^1(\cdot,t)\|_{X^\beta}=\|(T(h)-I)T(t+n)\phi_{\kappa}^+\|_{\infty}\leq \frac{C_{\beta}h^{\beta}e^{-(t-t_0)}}{(t-t_0)^{\beta}}\|\phi_{\kappa}^+\|_{\infty},
\end{equation}
\begin{align}\label{2'-eq}
&\|I_{t_0}^2(\cdot,t+h)-I_{t_0}^2(\cdot,t)\|_{X^\beta}\nonumber\\
&\le   C_{\beta}h^{\beta}\int_{t_0}^t\frac{e^{-(t-s)}}{(t-s)^{\frac{1}{2}+\beta}}\Big(\frac{a_{\sup}+\kappa^2}{\kappa}
+\chi\mu\|\phi(\cdot,s)\|_{\infty}\Big)\|\Phi(\cdot,s;t_0,\phi_{\kappa}^+,\phi)\|_{\infty}ds\nonumber\\
& \,\, + C_{\beta}\int_{t}^{t+h}\frac{e^{-(t+h-s)}}{(t+h-s)^{\frac{1}{2}+\beta}}\Big(\frac{a_{\sup}+\kappa^2}{\kappa}
+\chi\mu\|\phi(\cdot,s)\|_{\infty}\Big)\|\Phi(\cdot,s;t_0,\phi_{\kappa}^+,\phi)\|_{\infty}ds\nonumber\\
&\leq  C_{\beta}\Big(h^{\beta}\Gamma(\frac{1}{2}-\beta)+\frac{h^{\frac{1}{2}-\beta}}{\frac{1}{2}-\beta} \Big)\Big(\frac{a_{\sup}+\kappa^2}{\kappa}+\frac{\chi\mu a_{\sup}}{b_{\inf}-\chi\mu}\Big)\frac{a_{\sup}}{b_{\inf}-\chi\mu},
\end{align}
\begin{equation}\label{3'-eq}
\|I_{t_0}^3(\cdot,t+h)-I_{t_0}^3(\cdot,t)\|_{X^\beta}\leq  C_{\beta}\Big(h^{\beta}\Gamma(1-\beta)+\frac{h^{1-\beta}}{1-\beta} \Big)\Big(a_{\sup}+1+\frac{\chi\mu a_{\sup}}{b_{\inf}-\chi\mu} \Big)\frac{a_{\sup}}{b_{\inf}-\chi\mu},
\end{equation}
and
\begin{equation}\label{4'-eq}
\|I_{t_0}^4(\cdot,t+h)-I_{t_0}^4(\cdot,t)\|_{X^\beta}\leq   C_{\beta}\Big(h^{\beta}\Gamma(1-\beta)+\frac{h^{1-\beta}}{1-\beta} \Big)\Big(b_{\sup}-\chi\mu \Big)\frac{a_{\sup}}{b_{\inf}-\chi\mu}.
\end{equation}
\eqref{b-eq2} then follows from \eqref{1'-eq}--\eqref{4'-eq}.

Letting $t_0\to -\infty$ in \eqref{b-eq2},
we obtain
\begin{equation}\label{b-eq2-1}
\|\Phi(\cdot,t+h;\phi)-\Phi(\cdot,t;\phi)\|_{X^{\beta}}\leq \bar{C}_{\beta}(h^{1-\beta}+h^{\frac{1}{2}-\beta}),  \ \forall h>0,\ t\in\R.
\end{equation}
Note that $X^{\beta}$ is continuously embedded in $C^{2\nu}_{\rm unif}(\R)$ for $0<\nu<2\beta$ (\cite{D_Henry}). This together with \eqref{b-eq1} implies that there is $\hat C_\beta>0$ such that
\begin{equation}
\label{b-eq1-1}
|\Phi(x+h,t;\phi)-\Phi(x,t;\phi)|\le \hat C_\beta |h|^{\beta_0}\quad \forall \, x,t,h\in\R,\, \, |h|\le 1.
\end{equation}
By \eqref{z-eq2}, \eqref{b-eq2-1}, and \eqref{b-eq1-1}, we have that  $|\Phi(\cdot,\cdot;\phi)\in \mathcal{\tilde E}_{\kappa,\beta_0,M}$ for every
$\phi\in \mathcal{\tilde E}_{\kappa,\beta_0,M}$ provided that $M\ge \max\{2\tilde C_\beta,\hat C_\beta\}$.

Finally, by  \eqref{b-eq1}, \eqref{b-eq2},   Arzela-Ascoli Theorem,  and \cite[Chapter 3 - Theorem 15]{Fri},  we have that
$\Phi(x,t;t_0,\phi) \to \Phi(x,t;\phi)$ in $C^{2,1}(\bar D)$ as $t_0\to -\infty$  for any bounded domain $D\subset\R\times\R$. This implies that $\Phi(\cdot,\cdot,\phi)$ satisfies \eqref{Main U-eq}.
\end{proof}

In the following, we fix $\frac{\beta_0}{2}<\beta<\frac{1}{2}-\beta_0$ and  $M\ge \max\{2\tilde C_\beta,\hat C_\beta\}$.


\begin{lem}\label{uniqueness-lem}
For every $\phi\in \mathcal{\tilde E}_{\kappa,\beta_0,M}$, the function $\Phi(\cdot,\cdot;\phi)$ given by \eqref{Main-U-def} is the only solution of \eqref{Main U-eq} in $\mathcal{\tilde E}_{\kappa,\beta_0,M}$.
\end{lem}

\begin{proof}
We apply the similar arguments as those  in \cite[Lemma 3.6]{SaSh2} to prove this lemma.

Let $\phi\in \mathcal{\tilde E}_{\kappa,\beta_0,M}$ be fixed and $U_1(x,t), U_2(x,t)$ be solutions of \eqref{Main U-eq} in $ \mathcal{\tilde E}_{\kappa,\beta_0,M}$. Let $\t>0$ be given. Since $\phi_{\kappa}^-(x,t)\leq U_{i}(x,t)\leq \phi_{\kappa}^+(x)$ for every $x\in\R$ and $t\in\R$, and $\lim_{x\to\infty}\frac{\phi_{\kappa}^+(x)}{\phi_{\kappa}^-(x,t)}=1$ uniformly in $t\in\R$, then there is $\bar{x}_{\tau}\gg 1$ such that
\begin{equation}\label{l-0}
\frac{1}{1+\t}\leq \frac{U_1(x,t)}{U_2(x,t)}\leq 1+\t, \quad \forall\ x\geq \bar{x}_{\t}, \ \forall\ t\in\R.
\end{equation}
Note that $\inf\{\phi_{\kappa}^{-}(x,t) \,| \, x\leq R , \, t\in\R\}>0$ for every $R\in\R$ and $\phi_{\kappa}^+(x)\leq \frac{a_{\sup}}{b_{\inf}-\chi\mu}$. This together with \eqref{l-0} and the fact that $\phi_{\kappa}^{-}(x,t)\leq U_i(x,t)\leq \phi_{\kappa}^{+}(x)$ for every $(x,t)\in\R\times\R$ imply that  there is $\alpha\in\R$ such that
\begin{equation}\label{l-00}
\alpha:=\inf\{\tilde{\alpha}\,| \, \tilde{\alpha}\geq 1 \ \text{and}\ \frac{1}{\tilde{\alpha}}\leq \frac{U_1(x,t)}{U_{2}(x,t)}\leq \tilde{\alpha} , \quad \forall\ x\in\R, \ \forall\ t\in\R\}.
\end{equation}
It is clear $\alpha\ge 1$, and to prove the lemma, it suffices to prove that $\alpha=1$. We prove this in the following by contradiction.

Assume that $\alpha>1$. It is clear from the definition of $\alpha$ that
\begin{equation}\label{l-01}
\begin{cases}
U_{1}(x,t)\leq \alpha U_{2}(x,t), \quad \forall\ x\in\R, \ t\in\R, \quad \text{and}\cr
U_{2}(x,t)\leq \alpha U_{1}(x,t), \quad \forall\ x\in\R, \ t\in\R.
\end{cases}
\end{equation}
 Let $\tau\in (0, \alpha-1)$ and $\bar{x}_{\tau}$ be given by \eqref{l-0}.
Note that
$$
U_i(x,t)=\Phi(x,t;t_0,U_i(\cdot,t_0),\phi),\quad \forall \ x\in\R, \ \forall\ t\geq t_0,
$$
where $\Phi(x,t;t_0,U_i,\phi)$ is given by \eqref{Un def} for each $i=1,2$.  Since $ \alpha>1$ and $U_i(x,t)>0$ for every $(x,t)\in\R\times\R$, then

\begin{equation}\label{l-002-0}
\mathcal{L}_{\kappa,\phi}(\alpha U_i)= (\alpha-1)(b(t)-\chi\mu)U_i(\alpha U_i)>0, \quad \forall\, x\in\R,\,\, t\in\R, \ \, i=1,2.
\end{equation}
Hence, since $\alpha>1$, comparison principle for parabolic equations and \eqref{l-01} imply that
\begin{equation}\label{l-002-1}
\begin{cases}
U_{1}(x,t+t_0)<\Phi(x,t+t_0;t_0,\alpha U_2,\phi)<\alpha U_{2}(x,t+t_0), \quad \forall\ x\in\R,\,\, \forall\, t>0, \,\,\forall\, t_0\in\R\cr
U_{2}(x,t+t_0)<\Phi(x,t+t_0;t_0,\alpha U_1,\phi)<\alpha U_{1}(x,t+t_0), \quad \forall\ x\in\R,\,\, \forall\, t>0, \,\,\forall\, t_0\in\R.
\end{cases}
\end{equation}

Note that \eqref{l-002-1} implies that  $$  \min_{\frac{1}{2}\leq t\leq 1}(\alpha U_{i}(\bar{x}_{\tau},t+t_0)-\Phi(\bar{x}_{\tau},t+t_0;t_0,\alpha U_{i},\phi))>0, \quad \forall\, t_0\in\R,\quad i=1,2.$$  We claim that for each $i=1,2$, we have that

 \begin{equation}\label{a-claim-1} \inf_{t_0\in\R}\min_{\frac{1}{2}\leq t\leq 1}(\alpha U_{i}(\bar{x}_{\tau},t+t_0)-\Phi(\bar{x}_{\tau},t+t_0;t_0,\alpha U_{i},\phi))>0.
\end{equation}
 Assume that \eqref{a-claim-1}  does not hold. Without loss of generality, we may suppose that there are sequences $\{t_{0n}\}_{n\geq 1} \subset \R$ and $\{t_n\}_{n\geq 1}\subset [\frac{1}{2}, 1]$ such that  \begin{equation}\label{a-eq-0-1-0} \lim_{n\to\infty}\alpha U_{1}(\bar{x}_{\tau},t_n+t_{0n})-\Phi(\bar{x}_{\tau},t_n+t_{0n};t_{0n},\alpha U_{1},\phi)=0.\end{equation} For each $n\geq 1$, let $u_{n}(x,t)=U_1(x+\tilde{x}_{\tau},t+t_n+t_{0n})$ and $\tilde{u}_{n}(x,t)=\Phi(x+\tilde{x}_{\tau},t+t_n+t_{0n};t_{0n},\alpha U_1,\phi)$ for $x\in\R$ and $t\ge-\frac{1}{2}$. Hence, by a priori estimates for parabolic equations and Arzela-Ascoli, we may suppose that there exist $u^*(x,t), \tilde{u}^*(x,t)\in C^{2,1}(\R\times[-\frac{1}{2},\infty))$, $\phi^*\in C^{\delta}(\R\times\R)$, and $a^*(t),b^*(t)\in C^{\delta}(\R)$ such that
  $$  (u_{n}(x,t),\tilde{u}_{n}(x,t),\phi(t+t_n+t_{0n}))\to (u^*(x,t),\tilde{u}^*(x,t),\phi^*(x,t)) $$
  and
  $$  (a(t+t_n+t_{0n}),b(t+t_n+t_{0n}))\to (a^*(t),b^*(t)) $$
  as $ n\to\infty$ in the compact-open topology.
 Moreover, $u^*(x,t)$  and $\tilde{u}^*(x,t)$ solve  the PDE
 \begin{equation}\label{a-eq-0-1-00}
 u_t=u_{xx} +(c^*(t)-\chi\partial_x\psi(\cdot,\cdot;\phi^*))u_{x}+ (a^*(t)-\chi\lambda\psi(\cdot,\cdot;\phi^*)-(b^*(t)-\chi\mu)u)u, \,\,  x\in\R,\,\, t\ge -\frac{1}{2}
 \end{equation}
 with $c^*(t)=\frac{\kappa^2+a^*(t)}{\kappa}$.
  Observe that
  $$
  0<\inf_{s\in\R}\phi_{\kappa}^{-}(x,s)\leq \min\{u^*(x,t),\tilde{u}^*(x,t)\}, \quad \forall\,   x\ll -1,\ t\ge -\frac{1}{2}, $$ and
  that
 \begin{equation*}
 \tilde{u}^*(x,t)\leq \alpha u^*(x,t), \quad \forall\, x\in\R,\,\, t\ge-\frac{1}{2}.
 \end{equation*}
 Then $0<u^*(x,t)$ for all $x\in\R$ and $t>\frac{1}{2}$. Multiplying \eqref{a-eq-0-1-00} by $\alpha$, we obtain that
  $$
  (\alpha u^*)_t>u_{xx} +(c^*(t)-\chi\partial_x\psi(\cdot,\cdot))(\alpha u^*)_{x}+ (a^*(t)-\chi\lambda\psi(\cdot,\cdot;\phi^*)-(b^*(t)-\chi\mu)(\alpha u^*))(\alpha u^*).
  $$
  Hence, comparison principle for parabolic equations yields that
  $$  u^*(x,t)<\alpha \tilde{u}^*(x,t),\quad \forall\ x\in\R,\ \forall\ t>-\frac{1}{2}. $$
  In particular
   \begin{equation}\label{a-eq-0-1-2}
   u^*(0,0)<\alpha \tilde{u}^*(0,0).
   \end{equation}
   On the other hand, \eqref{a-eq-0-1-0} implies that  $$ u^*(0,0)=\alpha\tilde{u}^*(0,0),$$
    which contradicts to \eqref{a-eq-0-1-2}.  Thus \eqref{a-claim-1} holds.

 We note  from \eqref{l-002-0} that  \begin{equation}\label{l-002} \begin{split}\mathcal{L}_{\kappa,\phi}(\alpha U_i)=& (\alpha-1)(b(t)-\chi\mu)U_i(\alpha U_i)\\\geq& (\alpha-1)(b_{\inf}-\chi\mu)m_{\tau}(\alpha U_i),\quad \forall\ x\leq \bar{x}_{\tau},\ \forall\ t\in\R.\end{split}\end{equation}where $m_\tau=\inf\{\phi_{\kappa}^-(x,t)\ :\ x\leq \bar{x}_{\tau},\,\, t\in\R\}>0$.


Let $0<\tilde{\delta}\ll 1$  such that
 $$
 \begin{cases}  \tilde{\delta}+(b_{\sup}-\chi\mu)(1-e^{-\tilde{\delta}})e^{\tilde{\delta}}\frac{\alpha a_{\sup}}{b_{\inf}-\chi\mu}\leq (\alpha-1)(b_{\inf}-\chi\mu)m_{\tau},\cr \inf_{t_0\in\R}\min_{\frac{1}{2}\leq t\leq 1}(\alpha U_{i}(\bar{x}_{\tau},t+t_0)-\Phi(\bar{x}_{\tau},t+t_0;t_0,\alpha U_{i},\phi))>\frac{(e^{\tilde{\delta}}-1)\alpha a_{\sup}}{b_{\inf}-\chi\mu}\cr
 (1+\tau)e^{\tilde{\delta}}<\alpha .
  \end{cases}
   $$
Let  $t_0\in\R$ be given. For every $t\ge 0$ and $i=1,2$,  let $W_{i}(x,t)=e^{\tilde{\delta} t}\Phi(x,t+t_0;t_0,\alpha U_i(\cdot,t_0),\phi)$.
 We have
 \begin{equation}\label{l-003} \begin{split} \mathcal{L}_{\kappa,\phi}(W_{i})(x,t)\leq & \Big(\tilde{\delta}+(b(t)-\chi\mu)(1-e^{-\tilde{\delta} t})e^{\tilde{\delta} t}W_{i}(x,t)\Big)W_i(x,t)\\ \leq & \Big(\tilde{\delta}+(b_{\sup}-\chi\mu)(1-e^{-\tilde{\delta} t})e^{\tilde{\delta} t}\frac{\alpha a_{\sup}}{b_{\inf}-\chi\mu}\Big)W_{i}(x,t)\\ \leq & (\alpha-1)(b_{\inf}-\chi\mu)m_{\tau}W_{i}(x,t),\ \quad \forall\,  0\leq t\leq 1. \end{split}
 \end{equation}
 Since $\Phi(x,t+t_0;t_0,\alpha U_i,\phi)\leq \alpha U_i(x,t+t_0)\leq \frac{\alpha a_{\sup}}{b_{\inf}-\chi\mu}$ for every $t\geq 0, t_0\in\R$ and $x\in\R$, then
 $$
\frac{(e^{\tilde{\delta}}-1)\alpha a_{\sup}}{b_{\inf}-\chi\mu} +\Phi(x,t+t_0;t_0,\alpha U_i,\phi)\geq e^{\tilde{\delta}t}\Phi(x,t+t_0;t_0,\alpha U_i,\phi),\quad \forall\ x\in\R,\ t_0\in\R, \ 0\leq t\leq 1.
 $$
 Therefore, it follows from the choice of $\tilde{\delta}$ that  $$  W_i(x,t)=e^{\tilde{\delta} t}\Phi(\tilde{x}_{\tau},t+t_0;t_0,\alpha U_i,\phi)\leq \alpha U_i(\tilde{x}_{\tau},t+t_0), \quad\,\,\,\, \forall\, t\in[\frac{1}{2}, 1], \,\,\forall\, t_0\in\R. $$
 Combine this with \eqref{l-002}, \eqref{l-003}, the comparison principle for parabolic equations yield that
 $$
  e^{\tilde{\delta} t}\Phi(x,t+t_0;t_0;\alpha U_i,\phi)\leq \alpha U_i(x,t+t_0),  \quad\, \forall\, x\le \tilde{x}_\tau,\,\, \forall\, t\in[\frac{1}{2}, 1], \,\,\forall\, t_0\in\R.
  $$
  This combined with \eqref{l-002-1} give
   $$
    U_{1}(x,t)\leq \alpha e^{-\tilde{\delta}} U_{2}(x,t) \quad \text{and}
    \quad  U_{2}(x,t)\leq \alpha e^{-\tilde{\delta}} U_{1}(x,t), \quad x\le \tilde{x}_\tau,\,\, t\in\R.
    $$
    Which is together with \eqref{l-0} and the fact that $(1+\tau)e^{\tilde{\delta}}\leq \alpha$  yield that $$  \frac{1}{\alpha e^{-\tilde{\delta}}}\leq \frac{U_1(x,t)}{U_{2}(x,t)}\leq \alpha e^{-\tilde{\delta}}, \,\,\, \forall\ t \in\R, \forall\ x\in\R. $$  Hence by definition of $\alpha$, we must have  $\alpha\leq \alpha e^{-\tilde{\delta}}$, which is impossible. Thus $\alpha =1$ and the lemma is proved.
\end{proof}

Now we present the proof of Theorem \ref{Main-tm1}

\begin{proof} [Proof of Theorem \ref{Main-tm1}(1)]

For any given  $ 0<\underline{c}<{c}_\chi^*$, let $0<\kappa< \kappa_{\chi}^*$ be such that $\underline{c}=\frac{\underline{a}+\kappa^2}{\kappa}$.
  We prove that \eqref{P0} has a transition front solution $(u(x,t),v(x,t))=(U(x-C_\kappa(t),t),V(x-C_{\kappa}(t),t))$ with $C_\kappa(t)=\int_0^tc_{\kappa}(s)ds$ and $c_{\kappa}(t)=\frac{a(t)+\kappa^2}{\kappa}$.  To prove this, we prove
 the following claim.

 \medskip

 \noindent {\bf Claim 1.}{\it  The mapping $\Phi:  \mathcal{\tilde E}_{\kappa,\beta_0,M} \to   \mathcal{\tilde E}_{\kappa,\beta_0,M}$, $\phi\mapsto \Phi(\cdot,\cdot,\phi)$,  has a fixed point $\phi^*$ satisfying $\Phi(\cdot,\cdot,\phi^*)=\phi^*$.
Furthermore, any element  $\phi^*\in \mathcal{\tilde E}_{\kappa,\beta_0,M}$  with the property $\Phi(\cdot,\cdot,\phi^*)=\phi^*$ also satisfies
\begin{equation*}
\lim_{x\to-\infty}|\phi^*(x,t)-u^*(t)|=0 \quad \text{and}\quad \lim_{x\to\infty}\frac{\phi^*(x,t)}{e^{-\kappa x}}=1, \quad \text{uniformly in }\ t,
\end{equation*}
where $u^{*}(t)$ is the only positive entire solution of \eqref{kpp-eq}.}

\medskip

Assume that the claim holds.  Let $U(x,t)=\phi^*(x,t)$ and $V(x,t)=\Psi(x,t;U)$. Then $$(u(x,t),v(x,t))=(U(x-C_\kappa(t),t),V(x-C_{\kappa}(t),t))$$
  where $C_\kappa(t)=\int_0^tc_{\kappa}(s)ds$ and $c_{\kappa}(t)=\frac{a(t)+\kappa^2}{\kappa}$,  is a transition front solution of\eqref{P0} connecting $(0,0)$ and $(u^*(t),v^*(t))$ satisfying \eqref{TW-eq}.
In the following, we prove that the claim holds.

We first prove that the mapping $\Phi$ has a fixed point. To this end,
let $C^b_{\rm unif}(\R\times\R)$ be endowed with the compact-open topology, that is, a sequence of elements of $C^b_{\rm unif}(\R\times\R)$ converges if and only it converges  uniformly on every compact subsets of $\R\times\R$. It is clear that $\mathcal{\tilde E}_{\kappa,\beta_0,M}$ is a compact and convex subset of  $C^b_{\rm unif}(\R\times\R)$ in the compact-open topology.
{ To show that the mapping $\Phi$ has a fixed point, it is then enough to show that $\Phi$ is continuous. We show this in the following.}

 Let $\{\phi_{n}\}_{n\geq 1}$ and $\phi$ be elements of $\mathcal{\tilde E}_{\kappa,\beta_0,M}$ such that $\phi_{n}(x,t)\to \phi(x,t)$ as $n\to\infty$ uniformly on every compact subsets of $\R\times\R$. Since $\mathcal{\tilde E}_{\kappa,\beta_0,M}$ is compact, there is
 a subsequence $\{n_k\}$ of $\{n\}$ and a function $\Phi(\cdot,\cdot)\in \mathcal{\tilde E}_{\kappa,\beta_0,M}$
  such that $\Phi(x,t;\phi_{n_k})\to \Phi(x,t)$ as $n_k\to\infty$ on every compact subset of $\R\times\R$. Since $\phi_{n_k}(x,t)\to \phi(x,t)$  as $n_k\to\infty$ on very compact subset of $\R\times\R$, it follows from Lemma \ref{pointwise-estimate on space derivative of v} (ii) that
$$
\psi(x,t;\phi_{n_k})\to \psi(x,t;\phi) \quad \text{and}\quad \partial_x\psi(x,t;\phi_{n_2})\to \partial_x \psi(x,t;\phi)\quad \text{as }\ n_k\to\infty,
$$
uniformly on every compact subset of $\R\times\R$, where $\psi(\cdot,\cdot;\phi)$ is given by \eqref{space-derivatie of v}. Therefore, it follows from \cite[Chapter 3 Theorem 3]{Fri} that  $\Phi(x,t)\in C^{2,1}(\R\times\R)$ and satisfies
$$
\Phi_{t}=\Phi_{xx}+(c_{\kappa}(t)-\psi_{x}(x,t;\phi))\Phi_{x}+(a(t)-\chi\lambda\psi(x,t;\phi)-(b(t)-\chi\mu)\Phi)\Phi, \  x\in\R, \ t\in\R.
$$
Therefore, it follows from Lemma \ref{uniqueness-lem} that $\Phi(x,t)=\Phi(x,t;\phi)$ for every $(x,t)\in\R\times\R$.
Hence the mapping $\Phi$ is continuous, and by Schauder's fixed point theorem, there is $\phi^*\in\mathcal{\tilde E}_{\kappa,\beta_0,M}$ such that
$$
\Phi(x,t;\phi^*)=\phi^*(x,t)\,\, \,  \forall\ (x,t)\in\R\times\R.
$$

Next, we prove \eqref{TW-eq2}.
 Suppose that  $\phi^*=\Phi(\cdot,\cdot,\phi^*)$.
Since
$$
\lim_{x\to\infty}\frac{\phi_{\kappa}^+(x)}{\phi_{\kappa}^-(x,t)}=1,
$$
uniformly in $t\in\R$ and $\phi_{\kappa}^-(x,t)\leq \phi^*(x,t)\leq \phi^+_{\kappa}(x)$, for every $x\in\R$ and $t\in\R$, we get that
$$
\lim_{x\to\infty}\frac{\phi^*(x,t)}{e^{-\kappa x}}=1, \quad \text{uniformly in t.}
$$
It remains to show that
\begin{equation}\label{z--02}
\lim_{x\to-\infty}|\phi^*(x,t)-u^*(t)|=0, \quad \text{uniformly in} \ t,
\end{equation}
where $u^*(t)$ is given by \eqref{kpp-eq}.

Suppose by contradiction that \eqref{z--02} does not hold.
Then, there exist a sequence of real numbers $\{x_n\}_{n\geq 1}$ with $\lim_{n\to}x_{n}=-\infty$, a sequence of real numbers $\{t_n\}_{n\geq 1}$, and a positive real number $\delta>0$ such that
\begin{equation}\label{z--03}
\delta\leq \inf_{n\geq 1}|\phi^*(x_n,t_n)-u^*(t_n)|.
\end{equation}
Let $\psi^*(x,t)=\psi(x,t;\phi^{*})$. Hence $(u(x,t),v(x,t)):=(\phi^*(x-\int_0^tc_\kappa(s)ds,t),\psi^*(x-\int_0^tc_\kappa(s)ds,t))$ solves \eqref{P0}.

For each $n\geq 1$, $x\in\R$, and $t\in\R$, let
$$
u_{n}(x,t)=u(x+x_n+\int_0^{t_n}c_{\kappa}(s)ds,t+t_n), \quad u^*_n(t)=u^*(t+t_n),\quad a_n(t)=a(t+t_n),\quad b_n(t)=b(t+t_n),
$$
and $v_n=\mu(\lambda I-\Delta)^{-1}u_n$.  Note $(u_n,v_n)$ solves
\begin{equation*}
\begin{cases}
\p_t u_n=\p_{xx}u_n-\chi\p_x(u_n\p_xv_n)+(a_n(t)-b_n(t)u_n)u_n, \quad x\in\R,\ t\in\R\cr
0=\p_{xx}v_n-\lambda v_n +\mu u_n, \quad x\in\R,\ t\in\R,
\end{cases}
\end{equation*}
and
\begin{equation*}
\frac{d}{dt}u^*_n=(a_n(t)-b_n(t)u_n)u_n, \quad t\in\R.
\end{equation*}
Hence, up to a subsequence, using similar arguments as those used in the proof of Lemma \ref{U-eq lem}, we may assume that there exist $(\tilde{u},\tilde{v})\in C^{2,1}(\R\times\R)$, $\tilde{u}^*\in C^{1}(\R)$, $\tilde{a},\tilde{b}\in C^{\nu_0}_{\rm unif}(\R)$ such that $(u_n,v_n)\to (\tilde{u},\tilde{v})$, $u^*_n\to\tilde{u}^*$, $a_n\to\tilde{a}$, and $b_n\to\tilde{b}$ as $n\to \infty$ locally uniformly. Furthermore, it holds that
\begin{equation}\label{z--00}
\begin{cases}
\p_t\tilde{u}=\p_{xx}\tilde{u}-\chi\p_x(\tilde u\p_x\tilde v)+(\tilde a(t)-\tilde b(t)\tilde u)\tilde u, \quad x\in\R,\ t\in\R\cr
0=\p_{xx}\tilde v-\lambda\tilde v +\mu\tilde u, \quad x\in\R,\ t\in\R,
\end{cases}
\end{equation}
and
\begin{equation}\label{z--01}
\frac{d}{dt}\tilde{u}^*=(\tilde{a}(t)-\tilde{b}(t)\tilde{u}^*)\tilde{u}^*, \quad t\in\R.
\end{equation}
It is clear that $0<\inf_{t\in\R}\tilde{u}^*(t)\leq \sup_{t\in\R}\tilde{u}^*(t)<\infty$. Using the fact that $\inf_{t\in\R}x_{\kappa,\varepsilon,d,A}(t;\delta_0)\geq \inf_{t\in\R}x_{\kappa,\varepsilon,d,A}^-(t)>0$, it follows from
 \eqref{time dependent lower - solution} that
\begin{equation*}
\begin{split}
 u_n(x,t)=&\phi^*(x+x_n-\int_{t_n}^{t_n+t}c_{\kappa}(s)ds,t+t_n)\\
 \geq &  \phi_{\kappa}^{-}(x+x_n-\int_{t_n}^{t_n+t}c_{\kappa}(s)ds,t+t_n )\\
 \geq & \delta_0
\end{split}
\end{equation*}
whenever $ x+x_n-\int_{t_n}^{t_n+t}c_{\kappa}(s)ds\leq 0$ . Since $x_{n}\to-\infty$ as $n\to\infty$, and $\|c_{\kappa}\|_{\infty}\leq \frac{a_{\sup}+\kappa^2}{\kappa}$, then
$$
0<\delta_0\leq \tilde{u}(x,t)\leq \frac{a_{\sup}}{b_{\inf}-\chi\mu}, \quad \forall x\in\R,\ \forall\ t\in\R.
$$
Hence $(\tilde{u},\tilde{v})$ is a positive entire solution of \eqref{z--00}.  But, \cite[Theorem 1.4]{SaSh_6_II} implies that $(\tilde{u}^*,\tilde{v}^*)$, where $\tilde{v}^*=\mu(\lambda I -\Delta )^{-1}\tilde{u}^*$, is the only strictly positive entire solution \eqref{z--00}, when $0<\mu\chi\leq \frac{b_{inf}}{2}$. Thus we must have
\begin{equation}\label{z--4}
\tilde{u}(x,t)=\tilde{u}^*(t), \quad \forall\ x\in\R, \ \forall\ t\in\R.
\end{equation}
But inequality \eqref{z--03} gives that
$$
0<\delta\leq |\tilde{u}(0,0)-\tilde{u}^*(0)|.
$$
This contradicts to \eqref{z--4}. Therefore, Claim 1 holds and Theorem \ref{Main-tm1}(1) is thus proved.
\end{proof}

 \begin{proof} [Proof of Theorem \ref{Main-tm1}(2)]
 For given $c>c_\chi^*$, let $\kappa\in(0, \min\{\kappa_\chi,\sqrt{\hat{a}}\})$ $(\hat a=\frac{1}{T}\int_0^T a(t)dt$)  be such that $c=\frac{\hat{a}+\kappa^2}{\kappa}$.
To prove Theorem \ref{Main-tm1}(2), we first change the set $\mathcal{E}_{\kappa,\beta_0,M}$ to
 \begin{align}
\label{def-of-set-e-1}
\mathcal{E}^T_{\kappa,\beta_0,M}=\{ \phi\in C(\R, C_{\rm unif}^b(\R))\,|& \,  0\leq \phi(x,t)\leq \phi^+_{\kappa}(x), \,\, \phi(x,t+T)=\phi(x,t)\nonumber\\
& |\phi(x+h,t)-\phi(x,t)|\le M |h|^{\beta_0},\,\, {\rm and}\nonumber\\
& |\phi(x,t+h)-\phi(x,t)|\le M |h|^{\beta_0} \,\, \forall\, x,t,h\in\R,\, |h|\le 1\},
\end{align}
and change $\tilde{\mathcal{E}}_{\kappa,\beta_0,M}$ to
\begin{equation}\label{E-kappa-1}
\tilde{\mathcal{E}}^T_{\kappa,\beta_0,M}=\{\phi\in\mathcal{E}^1_{\kappa,\beta_0,M}\,|\, \phi_{\kappa}^-(x,t)\leq \phi(x,t)\leq \phi_{\kappa}^+(x), \quad \forall\ x\in\R, \ \forall\ t\in\R\}.
\end{equation}
Note that by the periodicity of $a(t)$ and $b(t)$, $A(t)$ in \eqref{A0 eq} can be chosen to be periodic in $t$ with period $T$ and then
 $\phi_{\kappa}^{-}(x,t)$  is periodic in $t$ with period $T$.

 Next,   note that for any $\phi\in \tilde{\mathcal{E}}^T_{\kappa,\beta_0,M}$ and any $t_1>t_2$,
 $$0\leq \Phi(x,t;t_2,\phi_{\kappa}^+,\phi)\leq \Phi(x,t;t_1,\phi_{\kappa}^+,\phi)\leq \phi_{\kappa}^+(x)$$
  for every $x\in\R$ and $ t\geq t_1$. Let
  $$
  \Phi(x,t;\phi)=\lim_{t_0\to\infty} \Phi(x,t;t_0,\phi_{\kappa}^+,\phi)=\lim_{n\to\infty}\Phi(x,t;-nT,\phi_\kappa^+,\phi).
  $$
By the periodicity of $a(t)$ and $b(t)$, we have
\begin{align*}
\Phi(x,t;\phi)&=\lim_{n\to\infty}\Phi(x,t;-nT,\phi_\kappa^+,\phi)\\
&=\lim_{n\to\infty} \Phi(x,t+nT;0,\phi_{\kappa}^+,\phi)\\
&=\lim_{n\to\infty}\Phi(x,t+(n+1)T;0,\phi_\kappa^+,\phi)\\
&=\Phi(x,t+T;\phi).
\end{align*}
Hence $\Phi(\cdot,\cdot;\phi)\in \tilde{\mathcal{E}}^T_{\kappa,\beta_0,M}$.

Now, by the similar arguments as in Claim 1, we can prove the following claim.

\medskip

 \noindent {\bf Claim 2.}{\it  The mapping $\Phi:  \mathcal{\tilde E}^T_{\kappa,\beta_0,M} \to   \mathcal{\tilde E}^T_{\kappa,\beta_0,M}$, $\phi\mapsto \Phi(\cdot,\cdot,\phi)$,  has a fixed point $\phi^*$ satisfying $\Phi(\cdot,\cdot,\phi^*)=\phi^*$.
Furthermore, any element  $\phi^*\in \mathcal{\tilde E}_{\kappa,\beta_0,M}$  with the property $\Phi(\cdot,\cdot,\phi^*)=\phi^*$ also satisfies
\begin{equation*}
\lim_{x\to-\infty}|\phi^*(x,t)-u^*(t)|=0 \quad \text{and}\quad \lim_{x\to\infty}\frac{\phi^*(x,t)}{e^{-\kappa x}}=1, \quad \text{uniformly in }\ t,
\end{equation*}
where $u^{*}(t)$ is the only positive entire solution of \eqref{kpp-eq}.}

\medskip

Finally, let  $\tilde U(x,t)=\phi^*(x,t)$ and $\tilde V(x,t)=\Psi(x,t;\tilde U)$. Then $$(u(x,t),v(x,t))=(\tilde U(x-C_\kappa(t),t),\tilde V(x-C_{\kappa}(t),t))$$
  where $C_\kappa(t)=\int_0^tc_{\kappa}(s)ds$ and $c_{\kappa}(t)=\frac{a(t)+\kappa^2}{\kappa}$,  is a transition front solution of\eqref{P0} connecting $(0,0)$ and $(u^*(t),v^*(t))$ satisfying \eqref{TW-eq} and \eqref{TW-eq1}.
Let
$$
U(x,t)=\tilde U(x+ct-\int_0^t c_k(s)ds,t)\quad {\rm and}\quad V(x,t)=\tilde V(x+ct-\int_0^t c_k(s)ds,t).
$$
Note that the function $\R\ni t\mapsto ct -\int_0^tc_{\kappa}(s)ds$ is periodic with period $T$.  Then
$$
(u(x,t),v(x,t))=(U(x-ct,t),V(x-ct,t))
$$
 is a (periodic) transition front solution of\eqref{P0} connecting $(0,0)$ and $(u^*(t),v^*(t))$ satisfying \eqref{TW-eq1} and \eqref{TW-eq2}.
\end{proof}

\begin{proof} [Proof of Theorem \ref{Main-tm1}(3)]
Suppose that $a(t)\equiv a$ and $b(t)\equiv b$ are independent of $t$. For every $c> {c}^*_\chi$,  let $\kappa\in(0, \min\{\kappa_\chi,\sqrt{{a}}\})$ be such that $c=\frac{{a}+\kappa^2}{\kappa}$.
Similarly, to prove Theorem \ref{Main-tm1}(3), we  change the set $\mathcal{E}_{\kappa,\beta_0,M}$ to
 \begin{align}
\label{def-of-set-e-2}
\mathcal{E}^0_{\kappa,\beta_0,M}=\{ \phi\in C_{\rm unif}^b(\R)\,|& \,  0\leq \phi(x)\leq \phi^+_{\kappa}(x) \,\, {\rm and}\,\, |\phi(x+h)-\phi(x)|\le M |h|^{\beta_0}\,\,\forall\, x,t,h\in\R,\, |h|\le 1\},
\end{align}
and change $\tilde{\mathcal{E}}_{\kappa,\beta_0,M}$ to
\begin{equation}\label{E-kappa-2}
\tilde{\mathcal{E}}^0_{\kappa,\beta_0,M}=\{\phi\in\mathcal{E}^2_{\kappa,\beta_0,M}\,|\, \phi_{\kappa}^-(x,t)\leq \phi(x)\leq \phi_{\kappa}^+(x), \quad \forall\ x\in\R, \ \forall\ t\in\R\}.
\end{equation}
Note that  $A(t)$ in \eqref{A0 eq} can be chosen to be independent of $t$ and then
 $\phi_{\kappa}^{-}(x,t)$  is independent of $t$.

Note that, for any $T>0$,
$$
\tilde{\mathcal{E}}^0_{\kappa,\beta_0,M}\subset \tilde{\mathcal{E}}^T_{\kappa,\beta_0,M}
$$
and $\tilde{\mathcal{E}}^0_{\kappa,\beta_0,M}$ is a closed convex subset of $\tilde{\mathcal{E}}^T_{\kappa,\beta_0,M}$
with respect to the open compact topology.

 Note also that,  for any $\phi\in \tilde{\mathcal{E}}^0_{\kappa,\beta_0,M}$ and any $t_1>t_2$, we have
 $$0\leq \Phi(x,t;t_2,\phi_{\kappa}^+,\phi)\leq \Phi(x,t;t_1,\phi_{\kappa}^+,\phi)\leq \phi_{\kappa}^+(x)$$
  for every $x\in\R$ and $ t\geq t_1$. Let
  $$
  \Phi(x,t;\phi)=\lim_{t_0\to\infty} \Phi(x,t;t_0,\phi_{\kappa}^+,\phi).
  $$
We have
\begin{align*}
\Phi(x,t;\phi)&=\lim_{t_0\to -\infty}\Phi(x,t;t_0,\phi_\kappa^+,\phi)\\
&=\lim_{t_0\to -\infty} \Phi(x,t-t_0;0,\phi_{\kappa}^+,\phi)\\
&=\lim_{t_0\to -\infty}\Phi(x, -t_0;0,\phi_\kappa^+,\phi)\\
&=\Phi(x,0;\phi)\quad \forall\, t\in\R.
\end{align*}
Hence $\Phi(x,t;\phi)\equiv \Phi(x;\phi)$ and $\Phi(\cdot;\phi)\in \tilde{\mathcal{E}}^0_{\kappa,\beta_0,M}$.

Then  $\Phi$ in {\bf Claim 2} has a fixed point $\phi^*\in \tilde{\mathcal{E}}^0_{\kappa,\beta_0,M}$.
Let  $U(x,t)=\phi^*(x)$ and $V(x)=\Psi(x,t;\tilde U)$. Then $$(u(x,t),v(x,t))=( U(x-ct), V(x-ct))$$
    is a traveling wave solution of\eqref{P0} connecting $(0,0)$ and $(\frac{a}{b},\frac{\mu}{\lambda}\frac{a}{b})$ satisfying \eqref{TW-eq3}.
\end{proof}

\section{Nonexistence of transition front solutions}

In this section, we prove Theorem \ref{Main-tm2} on the nonexistence of transition front solutions of \eqref{P0}
with least mean speed small then $2\sqrt{\underbar{a}}$.

\begin{proof}[Proof of Theorem \ref{Main-tm2}]
Suppose that $(u(t,x),v(t,x))=(U(x-C(t),t),V(x-C(t),t))$ is a transition front solution of \eqref{P0}
connecting $(u^*(t),\frac{\mu}{\lambda}u^*(t))$ and $(0,0)$.
It suffices to prove that
\begin{equation}
\label{EE-0}
2\sqrt{\underline{a}}\leq \liminf_{t-s\to\infty}\frac{C(t)-C(s)}{t-s}.
\end{equation}
We prove this in five steps.

\smallskip

\noindent {\bf Step 1.} In this step, we get some pointwise estimate for $v(t,x)$ and
$v_x(t,x)$.

Note that
\begin{equation}\label{E0-1}
0<m_0:=\inf_{x\leq 0, t\in\R}u(t,x+C(t))=\inf_{x\le 0, t\in\R}U(x,t),
\end{equation}
and that
\begin{equation}
\label{E0-2}
0\leq u(t,x)\leq M:=\frac{a_{\sup}}{ b_{\inf}-\chi\mu},\quad  \forall\ t,x\in\R.
\end{equation}
Note also that for every $R\gg 1$, there is $C_R\gg 1$ and $\varepsilon_R>0$ such that
 \begin{equation}\label{E2}
|\chi v_x(t,\cdot)|_{L^{\infty}(B_{\frac{R}{2}}(x_0))} + |\chi\lambda v(t,\cdot)|_{L^{\infty}(B_{\frac{R}{2}}(x_0))}\leq C_{R}\|u(t,\cdot)\|_{L^{\infty}(B_{R}(x_0))}+\varepsilon_R M, \quad \forall\ x_0\in\R
\end{equation}
with $\lim_{R\to\infty}\varepsilon_R=0$.
By the arguments of \cite[Lemma 2.2]{HaHe},  for every $p>1$, $t_0>0$, $s_0\geq 0$ and $R>0$, there is $C_{t_0,s_0,R,M,p}$
\begin{equation}\label{E1}
u(t+t_{1},x)\le C_{t_0,s_0,R,M,p}[u(t+s+t_1,y)]^{\frac{1}{p}}(M+1),\quad \forall\ s\in[0,s_0], \ t\geq t_0, \ t_1\in\R, \ |x-y|\leq R.
\end{equation}
 By \eqref{E2} and \eqref{E1} with $p>1$, $s_0=0$ and $t_0=1$, we have
\begin{align}
\label{E2-1}
|\chi v_x(t,x)|+\chi \lambda v(t,x)&= |\chi v_x(t_0+t-t_0,x)|+\chi \lambda v(t_0+t-t_0,x)\nonumber\\
&\le C_{R,p} \big(u(t_0+t-t_0,x)\big)^{\frac{1}{p}}+\varepsilon_R M\nonumber\\
&= C_{R,p} \big(u(t,x)\big)^{\frac{1}{p}}+\varepsilon_R M \quad \forall\,\, t\in\R,\,\, x\in\R,
\end{align}
where  $C_{R,p}=C_R\cdot  C_{1,0,R,M,p}\cdot (M+1) (>0)$.

\smallskip

\noindent {\bf Step 2.} In this step, we construct super-solutions of some equation related to the first equation in \eqref{P}.

By \cite[Lemma 2.2]{SaSh_7_I}, for any
 $0<a_0<\underline{a}$, there is $A_{a_0}(t)\in  W^{1,\infty}_{\rm loc}(\R)\cap L^{\infty}(\R)$ such that
$$
a(t)+ A_{a_0}'(t)\geq a_0 ,\quad \forall\ t\in\R.
$$
Fix $0<a_0<\underline{a}$.
For every $\gamma\in\R$ and $s\in\R$, let
$$\tilde{u}^{s,\gamma}(t,x)=e^{A_{a_0}(t)}u(t,x-\gamma s+C(s))$$
 for  $t\geq s$ and $x\in\R$. Then
\begin{align*}
\tilde{u}^{s,\gamma}_t&=A_{a_0}'(t)\tilde{u}^{s,\gamma}+e^{A_{a_0}(t)}u^{s,\gamma}_t\cr
&=\tilde{u}^{s,\gamma}_{xx}-\chi v_{x}(t,x-\gamma s+C(s))\tilde{u}^{s,\gamma}_x \cr
&\,\,\, +\tilde{u}^{s,\gamma}(a(t)+ A_{a_0}'(t)-\chi\lambda v(t,x-\gamma s+C(s))-(b(t)-\chi\mu)u)\cr
&\ge  \tilde{u}^{s,\gamma}_{xx}-\chi v_{x}(t,x-\gamma s+C(s))\tilde{u}^{s,\gamma}_x \cr
&\,\,\, +\tilde{u}^{s,\gamma}(a_0-\chi\lambda v(t,x-\gamma s+C(s))-(b_{\sup}-\chi\mu)e^{\|A_{a_0}\|_{\infty}}\tilde{u}^{s,\gamma})\cr
&\ge  \tilde{u}^{s,\gamma}_{xx}-\chi v_{x}(t,x-\gamma s+C(s))\tilde{u}^{s,\gamma}_x \cr
&\,\,\, +\tilde{u}^{s,\gamma}(a_0-\chi \lambda \varepsilon_R M- \chi\lambda  C_{R,p} e^{\frac{1}{p}\|A_{a_0}\|_\infty} \big(\tilde u^{s,\gamma}\big)^{\frac{1}{p}}
 -(b_{\sup}-\chi\mu)e^{\|A_{a_0}\|_{\infty}}\tilde{u}^{s,\gamma}).
\end{align*}
Let $b_0=(b_{\sup}-\chi\mu)e^{\|A_{a_0}\|_{\infty}}$ and
\begin{align*}
\mathcal{A}_0(u)&=u_{xx}-\chi v_x(t,x-\gamma s +C(s))u_x\\
&\,\,\,  +u(a_0-\chi \lambda \varepsilon_R M- \chi\lambda C_{R,p} e^{\frac{1}{p}\|A_{a_0}\|_\infty} u^{\frac{1}{p}}
 -(b_{\sup}-\chi\mu)e^{\|A_{a_0}\|_{\infty}} u).
\end{align*}
Then $\tilde{u}^{s,\gamma}$ is a super-solution of
\begin{equation}
\label{E4-1}
{u}_t= \mathcal{A}_0 {u}
\end{equation}
for $t\ge s$ and $x\in\R$.

\smallskip

\noindent {\bf Step 3.} In this step, we construct sub-solutions of some equation related to the first equation in \eqref{P}.

For any $0<\gamma<2\sqrt{a_0}$,  choose $0<\varepsilon\ll 1$ such that
$$
4(a_0-4\varepsilon)-(\gamma+\varepsilon)^2\geq 2\varepsilon.
$$
By \cite[Lemma A1, Appendix]{NaRo1}, there exist is $r>0$ and  $h\in C^2(\R)$ satisfying that
\begin{equation*}
h(\rho)=0\,\, \text{for}\,\, \rho\in(-\infty, 0], \,\, h(\rho)=1\,\, \text{for}\, \,\rho\in[r,\infty),\, h'(\rho)>0 \,\, \text{for}\,\, 0<\rho<r
\end{equation*}
and that
\begin{equation}\label{E4-1-1}
-h''(\rho)+(\gamma+\varepsilon) h'(\rho)-(a_0-2\varepsilon)h<0 \quad \text{for}\,\, 0<\rho<r.
\end{equation}
Choose $R_0\gg 1$ and $0<\sigma\ll 1$ such that
$$
C_{R_0,p}\sigma^{\frac{1}{p}}+\varepsilon_{R_0}M<\varepsilon,
$$
$$
\sigma<m_0 e^{-\|A_{a_0}\|_\infty},
$$
and
$$\chi \lambda \varepsilon_{R_0} M+ \chi\lambda C_{R_0,p} e^{\frac{1}{p}\|A_{a_0}\|_\infty} \sigma^{\frac{1}{p}}
 +(b_{\sup}-\chi\mu)e^{\|A_{a_0}\|_{\infty}} \sigma<2\varepsilon.
 $$
Let
$$\underline{u}(t,x)=\sigma h(\gamma t -x)
$$
and
$$
B_s(t,x)=\min\{-\chi v_x(t,x-\gamma s+C(s)),\varepsilon\}.
$$
Then
\begin{align*}
\underline{u}_t=& \sigma \ga  h'(\ga t-x)  \cr
=& \sigma(\ga+\varepsilon)h'(\ga t -x)-\varepsilon\sigma  h'(\ga t-x)\cr
=&(\ga+\varepsilon)\sigma h'(\ga t -x) +\varepsilon \underline{u}_x\cr
\le & \sigma h''(\gamma t-x)+(a_0-2\varepsilon)\sigma  h(\ga t-x)+\varepsilon \underline{u}_x \quad \text{(by } \eqref{E4-1-1})\cr
=&\underline{u}_{xx}+\varepsilon \underline{u}_x +(a_0-2\varepsilon)\underline{u}\cr
\le & \underline{u}_{xx}+\min\{\varepsilon,-\chi v_x(t,x-\gamma s+C(s))\underline{u}_x+(a_0-2\varepsilon)\underline{u}\quad \text{(since }\ \underline{u}_x\leq 0)\cr
\le & \underline{u}_{xx}+B_{s}(t,x)\underline{u}_x+\underbar{u}\Big(a_0-\chi \lambda \varepsilon_R M- \chi\lambda C_{R,p} e^{\frac{1}{p}\|A_{a_0}\|_\infty} \underbar{u}^{\frac{1}{p}}
 -(b_{\sup}-\chi\mu)e^{\|A_{a_0}\|_{\infty}} \underbar{u}\Big).
\end{align*}
where $B_s(t,x)=\min\{\varepsilon,-\chi v_x(t,x-\gamma s+C(s))$.
Let
$$
\mathcal{\tilde A}_0u={u}_{xx} { + B_s(t,x)} {u}_x+{u}\Big(a_0-\chi \lambda \varepsilon_R M- \chi\lambda C_{R,p} e^{\frac{1}{p}\|A_{a_0}\|_\infty} {u}^{\frac{1}{p}}
 -(b_{\sup}-\chi\mu)e^{\|A_{a_0}\|_{\infty}} {u}\Big).
 $$
 Then $\underbar u(t,x)$ is a sub-solution of
 \begin{equation}
\label{E4-2}
{u}_t= \mathcal{\tilde A}_0 {u}
\end{equation}
for $t\ge s$ and $x\in\R$.

 Observe that \eqref{E4-1} and \eqref{E4-2} are the same at $(t,x)$ such that $B_s(t,x)= -\chi v_x(t,x-\gamma s+C(s))$,
and hence \eqref{E4-1} and \eqref{E4-2} are the same at $(t,x)$ such that $u(t,x)$ is sufficiently { small}.

\smallskip

\noindent {\bf Step 4.} In this step, we prove $\underbar{u}(t,x)\le \tilde{u}^{s,\gamma}(t,x)$ for $t\ge s$ and $x\in\R$.

First of all, for any given $t_0>s$, there are $x_0^-<x_0^+$ such that
\begin{equation}
\label{new-eq1}
\tilde{u}^{s,\gamma}(t,x)>\sigma\ge \underbar{u}(t,x) \quad {\rm for}\quad s\le t\le t_0,\,\, x\le x_0^-
\end{equation}
and
\begin{equation}
\label{new-eq2}
\tilde{u}^{s,\gamma}(t,x)<\sigma,\quad { \chi v_x(t,x-\gamma s+C(s))=-B_s(t,x)} \quad {\rm for}\quad s\le t\le t_0,\,\, x\ge x_0^+.
\end{equation}
Note that
\begin{equation}
\label{new-eq2-1}
 \underbar{u}(s,x)<\tilde{u}^{s,\gamma}(s,x) \quad {\rm \forall}\,\, x\in\R.
\end{equation}
Hence, there is $s<\tilde t_0\le t_0$ such that
$$
\underbar{u}(t,x_0^+)\le \tilde{u}^{s,\gamma}(t,x_0^+)\quad {\rm for}\quad s\le t\le \tilde t_0.
$$
Then by \eqref{new-eq2} and comparison principle for parabolic equations, we have
\begin{equation}
\label{new-eq3}
\underbar{u}(t,x)\le \tilde{u}^{s,\gamma}(t,x)\quad {\rm for}\quad s\le t\le \tilde t_0,\,\, x\ge x_0^+.
\end{equation}

Next, we prove that
\begin{equation}
\label{new-eq3-1}
\underline{u}(t,x)\leq \tilde{u}^{s,\gamma}(t,x)\quad {\rm for}\quad s\le t\le t_0,\,\,  x_0^- \le x\le x_0^+.
\end{equation}
Assume on the contrary that there is  $s<\tilde t_1\le t_0$ and $x_0^- \le\tilde x_1\le x_0^+$  such that
$$
\underline{u}(\tilde t_1,\tilde x_1)>\tilde{u}^{s,\gamma}(\tilde t_1,\tilde x_1).
$$
Let $$t_1:=\inf\{t\in(s,t_0]\,|\, \exists\ x_t\in[x_0^-,x_0^+]\ \text{satisfying} \ \underline{u}(t,x_t)\ge \tilde{u}^{s,\gamma}(t,x_t) \}. $$
 Note that $t_1>s$, otherwise there  would exist a sequence $\{(t_n,x_n)\}_{n\geq 1}$, $t_n\in(s,t_0]$, $x_n\in[x_0^-,x_0^+]$, and $\underline{u}(t_n,x_n)\ge\tilde{u}^{s,\gamma}(t_n,x_n)$ for every $n$ with $t_n\to s$. Without loss of generality, we may suppose that $x_n\to x\in[x_0^-,x_0^+]$. Thus we must have that $\underline{u}(s,x)\geq \tilde{u}^{s,\gamma}(s,x)$, which contradicts to \eqref{new-eq2-1}. Therefore by \eqref{new-eq1}, \eqref{new-eq2-1}, and \eqref{new-eq3} we have that $t_1\le \tilde{t}_1$,
\begin{equation}\label{new-eq6}
\underbar{u}(t,x)<\tilde{u}^{s,\gamma}(t,x)\quad \forall\,\, s\le t<t_1,\,\,\, x\in\R,
\end{equation}
and
\begin{equation}\label{new-eq7}
\sigma\ge \underbar{u}(t_1,x_1)=\tilde{u}^{s,\gamma}(t_1,x_1)\quad\text{for some}\,\, x_0^-\le x_1\le x_0^+.
\end{equation}
This implies that there is $0<\delta\ll 1$ such that $s<t_1-\delta$ and
$$
B_s(t,x)=-\chi v_x(t,x-\gamma s+C(s))\quad {\rm for}\quad t_1-\delta\le t\le t_1,\,\, |x-x_1|\le \delta.
$$
Note that
$$
\underbar{u}(t_1-\delta,x)< \tilde{u}^{s,\gamma}(t_1-\delta,x)\quad \forall\,\, x\in\R
$$
and
$$
\underbar{u}(t,x_1\pm \delta)\le \tilde{u}^{s,\gamma}(t,x\pm \delta)\quad \forall\,\, t_1-\delta\le t\le t_1.
$$
Then by comparison principle for parabolic equations, we have
$$
\underbar{u}(t_1,x_1)<\tilde{u}^{s,\gamma}(t_1,x_1),
$$
which contradicts to \eqref{new-eq7}. Hence \eqref{new-eq3-1} holds.


Now,   by \eqref{new-eq2}, \eqref{new-eq3-1}, and comparison principle for parabolic equations, we have
$$
\underbar{u}(t,x)\le \tilde u^{s,\gamma}(t,x)\quad \forall\,\, s\le t\le t_0,\,\, x\in\R.
$$
Then, since $t_0>s$ is arbitrary, we have
\begin{equation}
\label{new-eq5}
\underbar{u}(t,x)\le \tilde{u}^{s,\gamma}(t,x)\quad \forall \,\, t\ge s,\,\, x\in\R.
\end{equation}

\smallskip

\noindent {\bf Step 5.} In this step, we prove that \eqref{EE-0} holds.

Taking $x=\gamma t-\frac{r}{2}$ in \eqref{new-eq5}, we obtain that
$$
0<e^{-\|A_{a_0}\|_{\infty}}h(\frac{r}{2})\leq u(t,\gamma(t-s)-\frac{r}{2} +C(s))=U(t,\gamma(t-s)-\frac{r}{2} +C(s)-C(t)), \forall\ t>s.
$$
This together with the fact that $\lim_{x\to\infty}U(t,x)=0$ uniformly in $t\in\R$ implies  that there is $L\gg 1$ such that
$$
\gamma\leq \frac{C(t)-C(s)}{t-s}+\frac{L+r/2}{t-s}\quad {\rm for}\quad t>s,
$$
and then
$$
\gamma\leq \liminf_{t-s\to\infty}\frac{C(t)-C(s)}{t-s}.
$$
Since, $\gamma$ is arbitrarily chosen less that $2\sqrt{a_0}$, we infer that
$$
2\sqrt{a_0}\leq \liminf_{t-s\to\infty}\frac{C(t)-C(s)}{t-s},\quad  \forall\ 0<a_0<\underline{a}.
$$
This implies that \eqref{EE-0} holds.
The theorem is thus proved.
\end{proof}

\end{document}